\documentclass{amsart}

\usepackage{amsmath}
\usepackage{amsthm}
\usepackage{amssymb}
\usepackage{url}

\usepackage{enumerate}
\usepackage[mathscr]{eucal}
\usepackage{mathrsfs}
\usepackage{verbatim}

\theoremstyle{plain}
\newtheorem{theorem}{Theorem}[section]
\newtheorem{lemma}[theorem]{Lemma}
\newtheorem{proposition}[theorem]{Proposition}

\newtheorem{claim}[theorem]{Claim}

\theoremstyle{remark}
\newtheorem{remark}[theorem]{Remark}

\numberwithin{equation}{section}
\numberwithin{theorem}{section}

\DeclareMathOperator{\supp}{supp}

\makeatletter
\@namedef{subjclassname@2010}{%
  \textup{2010} Mathematics Subject Classification}
\makeatother

\begin{document}


\date{}

\title[Bilinear Hilbert Transforms]{Bilinear Hilbert transforms associated to plane curves}

\author[J.W. Guo and L. Xiao]{Jingwei Guo \  \  Lechao Xiao}

\address{Jingwei Guo\\Department of Mathematics\\
University of Illinois at Urbana-Champaign\\
Urbana, IL 61801, USA}

\email{jwguo@illinois.edu}

\address{Lechao Xiao \\Department of Mathematics\\
University of Illinois at Urbana-Champaign\\
Urbana, IL 61801, USA}

\email{xiao14@illinois.edu}

\thanks{}

\begin{abstract}
We prove that the bilinear Hilbert transforms and maximal functions
along certain general plane curves are bounded from
$L^2(\mathbb{R})\times L^2(\mathbb{R})$ to $L^1(\mathbb{R})$.
\end{abstract}

\maketitle


%
%
%
%

\section{Introduction}\label{introduction}

Since the initial breakthroughs for singular integrals along curves
and surfaces by Nagel, Rivi\`{e}re, Stein, Wainger, et al., in the
1970s (see for example \cite{NRW1974-1976} and \cite{SW1978} for
some of their works on Hilbert transforms along curves), extensive
research in this area of harmonic analysis has been done and a great
many fascinating and important results have been established, which
culminate in a general theory of singular Radon transforms (see for
instance Christ, Nagel, Stein, and Wainger~\cite{CNSW1999}).

Another attractive area, parallel to the above one, is the bilinear
extension of the classical Hilbert transform. The boundedness of
such bilinear transforms was conjectured by Calder\'on and motivated
by the study of the Cauchy integral on Lipschitz curves. In the
1990s, this conjecture was verified by Lacey and Thiele in a
breakthrough pair of papers \cite{ML97, ML99}. In their works, a
systematic and delicate method was developed, inspired by the famous
works of Carleson \cite{CA66} and Fefferman \cite{Fe73}, which is
nowadays referred as the method of time-frequency analysis. Over the
past two decades, this method has merged as a powerful analytic tool
to handle problems that are related to multilinear analysis.

We are interested in the study of bilinear/multilinear singular
integrals along curves and surfaces--a problem that is closely
related to the two areas above. (We refer the readers to
Li~\cite{Li2013} for connections of this problem with ergodic theory
and multilinear oscillatory integrals.) To begin with, we consider a
model case--the truncated bilinear Hilbert transforms along plane
curves. One formulation of the problem is as follows.

Let $\Gamma(t)=(t, \gamma(t)):(-1, 1)\rightarrow \mathbb{R}^2$ be a
curve in $\mathbb{R}^2$. To $\Gamma$ we associate the truncated
bilinear Hilbert transform operator $H_{\Gamma}$ given by the
principal value integral
\begin{equation}
H_{\Gamma}(f, g)(x)=\int_{-1}^{1} \! f(x-t)g(x-\gamma(t)
)t^{-1}\,\mathrm{d}t  \qquad (x\in \mathbb{R}),
\label{bilinearHilbert}
\end{equation}
where $f$ and $g$ are Schwartz functions on $\mathbb{R}$. When the
function $\gamma$ has certain curvature (or ``non-flat'')
conditions, the boundedness properties of this operator ({\it e.g.}
whether it is bounded from $L^{p_1}(\mathbb{R})\times
L^{p_2}(\mathbb{R})$ to $L^r(\mathbb{R})$ for certain $p_1$, $p_2$,
and $r$) are of great interest to us.

Li~\cite{Li2013} studied such an operator (the integral defining
$H_{\Gamma}(f, g)(x)$ in \cite{Li2013} is over $\mathbb{R}$) with
the curve being defined by a monomial ({\it i.e.} $\gamma(t)=t^d$,
$d\in\mathbb{N}$, $d\geq2$) and proved that it is bounded from
$L^2(\mathbb{R})\times L^2(\mathbb{R})$ to $L^1(\mathbb{R})$. In his
proof, he combined results and tools from both time-frequency
analysis and oscillatory integral theory and used ingeniously a
uniformity concept (the so-called $\sigma$-uniformity; see
\cite[Section 6]{Li2013}). Lie~\cite{Lie2011} improved Li's results
both qualitatively, by extending monomials to more general curves
(certain ``slow-varying'' curves with extra curvature assumptions),
and quantitatively,  by improving the estimates. Instead of using
Li's method of $\sigma$-uniformity, Lie used a Gabor frame
decomposition to discretize certain operators in a smart way and
then worked with the discretized operators which have variables
separated on the frequency side and preserve certain main
characteristics (see the appendix of \cite{Lie2011} for a detailed
comparison between their methods).

Another interesting aspect of this problem was considered by Li and
the second author \cite{Li-Xiao}, in which they studied the case
when the curve is defined by a polynomial with different emphasis of
getting bounds uniform in coefficients of the polynomial and the
full range of indices $(p_1,p_2,r)$. They provided complete answers
(except to the endpoint case) for $H_\Gamma$ when the polynomial is
``non-flat''  near the origin ({\it i.e.} without the linear term).
When the polynomial has a linear term, however,  the full range of
indices for the corresponding uniform estimates is extremely
difficult to find and remains open.

In this paper we consider a family of general curves and provide an
easy-to-check criterion for a curve whose associated bilinear
Hilbert transform is bounded from $L^2(\mathbb{R})\times
L^2(\mathbb{R})$ to $L^1(\mathbb{R})$ (for the precise statement of
our results, see Section \ref{statement-theorem}). We use the
methods that are used in Li~\cite{Li2013} (with modifications).
Although we use different methods from Lie~\cite{Lie2011}, the
curves we consider are similar and the estimates we get are the
same.

Our criterion, motivated by results in Lie~\cite{Lie2011} and Nagel,
Vance, Wainger, and Weinberg~\cite{NVWW1983}, mainly asks one to
check whether certain bounds of various expressions involving
derivatives of a quotient are satisfied.  In \cite{NVWW1983} a
simple necessary and sufficient condition is provided (among other
results) for the $L^2$-boundedness of the Hilbert transform along
the curve $\Gamma$ with an odd function $\gamma$ and it is expressed
in terms of an auxiliary function $h(t)=t\gamma'(t)-\gamma(t)$. Both
Lie's and our results indicate that an appropriate replacement for
$h$ in the bilinear setting might be in the form of  a quotient (see
the $Q_{\epsilon}(t)$ defined in Section \ref{statement-theorem}).
We still do not know whether our criterion is a necessary condition.

In our main estimates in Section \ref{main-estimates}, we apply the
$TT^*$ method both in frequency space (with an extra size
restriction $|\gamma'(2^{-j})|>2^{-m}$) and in time space (with an
extra restriction on the function space), then we combine both
results to get the fast decay needed in proving the boundedness of
the desired operator. Since we are considering general curves, we
pay special attention to the dependence (on $\gamma$, $j$, $m$,
etc.) of all bounds we get, especially to those occurring when we
apply a quantitative version of the method of stationary phase.

We also establish analogous results for the bilinear maximal
function along $\Gamma$ (defined below) by using the arguments from
\cite[Section 7]{Li-Xiao} and our main estimates.
\begin{equation}
M_\Gamma(f,g)(x) =\sup_{0<\epsilon <1}\epsilon^{-1}
\int_{0}^\epsilon \! |f(x-t) g(x-\gamma(t))|\,\mathrm{d}t  \qquad
(x\in \mathbb{R}).  \label{mm99}
\end{equation}
We note that such an operator along a ``non-flat'' polynomial was
already carefully studied in \cite{Li-Xiao}. Much deeper and more
elegant results for a linear curve can be found in
Lacey~\cite{LA00}.

{\it Notations:} The Fourier transform of $f$ is
$\widehat{f}(\xi)=\mathcal{F}[f](\xi)=\int_{\mathbb{R}} \! f(x)
e^{-2\pi i\xi x} \,\textrm{d} x$ and its inverse Fourier transform
is $\mathcal{F}^{-1}[g](x)=\int_{\mathbb{R}} \! g(\xi) e^{2\pi i\xi
x} \,\textrm{d} \xi$. Let $\textbf{1}_{a, n}$ be the indicator
function of interval $a\cdot[n, n+1)$  for $a, n\in\mathbb{R}$ and
$1_{I}$ the indicator function of interval $I$. The indices
$(p_1,p_2,r)$ are always assumed to satisfy $1/p_1+1/p_2=1/r$,
$p_1>1$, $p_2>1$, and $r>1/2$. We use $C$ to denote an absolute
constant which may be depending on the curve and different from line
to line.

%
%
%
%
%
%

\section{Statement of theorems}\label{statement-theorem}

For any $a\in\mathbb{R}$, we say that a curve $\Gamma(t)=(t,
\gamma(t)+a):(-1, 1)\rightarrow \mathbb{R}^2$ (\footnote{In the
problems considered in this paper, we can always remove the constant
$a$ from the definition of $\Gamma$ by a translation argument, hence
there is no need to specify the dependence of $\Gamma$ on $a$ and we
will always let $a=0$.}) belongs to a family of curves,
$\textbf{F}(-1, 1)$, if the function $\gamma$ satisfies the
following conditions \eqref{origin-condition1}-\eqref{mt}. There
exists a constant $0<A_1<1/2$ such that on $(-A_1, A_1)\setminus
\{0\}$ the function $\gamma$ is of class $C^N$ ($N\geq 5$) and
$\gamma'\neq 0$. Let $Q_{\epsilon}(t)=\gamma(\epsilon
t)/\epsilon\gamma'(\epsilon)$. For $0<|\epsilon|<c_0<A_1/4$ and
$1/4\leq |t|\leq 4$, we have
\begin{equation}
|D^j Q_{\epsilon}(t)|\leq C_1, \quad 0\leq j\leq N,
\label{origin-condition1}
\end{equation}
\begin{equation}
|D^2 Q_{\epsilon}(t)|\geq c_1, \label{origin-condition2}
\end{equation}
(\footnote{The condition \eqref{origin-condition2} implies that
there exist constants $K_1, K_2>0$ such that
\begin{equation*}
|\gamma'(\epsilon)|\leq K_1|\epsilon|^{c_1}\quad \textrm{for
$0<|\epsilon|<c_0$}
\end{equation*}
or
\begin{equation*}
|\gamma'(\epsilon)|\geq K_2|\epsilon|^{-c_1}\quad \textrm{for
$0<|\epsilon|<c_0$}.
\end{equation*}
See also Lie~\cite[P.~4]{Lie2011} Observation (6) and (7).}) and
\begin{equation}
|(D^2Q_{\epsilon})^2(t)-D^1Q_{\epsilon}(t)D^3Q_{\epsilon}(t)|\geq
c_2, \quad \textrm{if $|\gamma'(\epsilon)|\leq
K_1|\epsilon|^{c_1}$,}\label{origin-condition3}
\end{equation}
or
\begin{equation}\tag{\ref{origin-condition3}$'$}
|2(D^2Q_{\epsilon})^2(t)-D^1Q_{\epsilon}(t)D^3Q_{\epsilon}(t)|\geq
c_3, \quad \textrm{if $|\gamma'(\epsilon)|\geq
K_2|\epsilon|^{-c_1}$.} \label{origin-condition3-prime}
\end{equation}
Let $\Delta_j=|2^{-j}\gamma'(2^{-j})|^{-1}$. If
$\gamma''(\epsilon)\gamma'(\epsilon)<0$ for $0<\epsilon<c_0$, then
there exist $K_3\in\mathbb{Z}$ and $K_4\in\mathbb{N}$ such that
\begin{equation}
\Delta_{j+K_3} \geq 2 \Delta_j, \quad \textrm{if $j\geq K_4$}.
\label{mt}
\end{equation}

\begin{theorem} \label{main-theorem}
If $\Gamma\in \textbf{F}(-1, 1)$, then $H_{\Gamma}(f, g)$ can be
extended to a bounded operator from $L^2(\mathbb{R})\times
L^2(\mathbb{R})$ to $L^1(\mathbb{R})$.
\end{theorem}

The analogous version for bilinear maximal functions is as follows.

\begin{theorem}\label{maxT}
If $\Gamma\in \textbf{F}(-1, 1)$, then $M_\Gamma(f,g)$ is a bounded
operator from $L^2(\mathbb{R})\times L^2(\mathbb{R})$ to
$L^1(\mathbb{R})$.
\end{theorem}

\begin{remark}
Combining the results in this paper with the time-frequency analysis
arguments in \cite{Li-Xiao},
the $L^{p_1}(\mathbb R)\times L^{p_2}(\mathbb R)\to L^{r}(\mathbb
R)$ boundedness of $H_\Gamma$ and $M_\Gamma$ for $r<1$ may be
obtained. We do not carry the details out in this paper. The lower
bound of such $r$, as indicated in \cite[Theorem 4]{Li-Xiao}, is
closely related to the decay rate of the following sublevel set
estimates
\begin{align}\label{to55}
|\{|t|<1: |\gamma'(t)-1| <h\}|,
\end{align}
as $h\to 0^+$. In particular, if (\ref{to55}) is controlled by
$c_\nu h^{\nu}$ for some $\nu>0$ and $c_\nu>0$, then $H_\Gamma$ and
$M_\Gamma$ are expected to be bounded from $L^{p_1}(\mathbb R)\times
L^{p_2}(\mathbb R)$ to $L^{r}(\mathbb R)$ given
$r>\max\{1/(1+\nu),1/2\}$; see \cite[Theorem 4]{Li-Xiao} when
$\gamma$ is a polynomial.

\end{remark}

\begin{remark}\label{remark2-1}
(1) We will use $A_1$, $c_0$, $c_1$, $c_2$, $c_3$, $C_1$, $K_1$,
$K_2$, $K_3$, and $K_4$ throughout this paper.

(2) The condition \eqref{origin-condition1} with $j=1$ implies that
\begin{equation*}
|D^1 Q_{\epsilon}(t)|\geq 1/C_1, \quad \textrm{for
$0<|\epsilon|<c_0/4$, $1/4\leq |t|\leq 4$}.
\end{equation*}

(3) If $\gamma''(\epsilon)\gamma'(\epsilon)>0$  for
$0<\epsilon<c_0$, then \eqref{mt} always holds with $K_3=1$.

(4) We now compare our assumptions
\eqref{origin-condition1}-\eqref{mt} with Lie's assumptions (1)-(5)
in \cite[P.~3]{Lie2011}. The \eqref{mt} implies Lie's (1). The
\eqref{origin-condition1} and \eqref{origin-condition2} correspond
to Lie's (2) and (4) (the $Q''$ part) while the
\eqref{origin-condition3} essentially corresponds  to Lie's (5).

(5) Note that the curves considered here are not necessarily
differentiable at the origin (they can even have a pole). One
explanation for this phenomenon is that the bilinear Hilbert
transform possesses certain symmetry between its two functions $f$
and $g$ (as well as its two variables $\xi$ and $\eta$ on the
frequency side) that we can take advantage of to essentially
transfer the case with a pole to the case without a pole (see the
two expressions of $B_{j, m}^{\varphi}(f, g)$ at the beginning of
Section \ref{main-estimates}).

\end{remark}

\begin{remark}
Here are some curves $\Gamma(t)=(t, \gamma(t))$ that belong to
$\textbf{F}(-1, 1)$:
\begin{enumerate}
\item Those smooth curves that have contact with $t$-axis at the
origin of finite order $\geq 2$ (namely,
$\gamma(0)=\gamma'(0)=\ldots=\gamma^{(d-1)}(0)=0$, but
$\gamma^{(d)}(0)\neq 0$ for some natural number $d\geq 2$), for
example, $\gamma(t)=t^d$ if $d\geq 2$;
\item The function $\gamma$ has a pole at the origin of finite order $\geq 1$ (namely, $\gamma(t)=t^{-n}h(t)$ for
 some natural number $n\geq 1$ and some smooth function $h$ with $h(0)\neq 0$);
\item $\gamma(t)=$\ a linear combination of finitely many terms of the form
$|t|^{\alpha}|\log|t||^{\beta}$ for $\alpha, \beta\in \mathbb{R}$
and $\alpha\neq 0, 1$;
\item $\gamma(t)={\rm sgn}(t)|t|^{\alpha}$ or $|t|^{\alpha}|\log|\log|t|||^{\beta}$ for $\alpha, \beta\in \mathbb{R}$
and $\alpha\neq 0, 1$.
\end{enumerate}
\end{remark}


%
%
%
%
%
%
\section{Preliminaries}\label{phase-function}

In this section, we first study a special oscillatory integral which
occurs in later sections. The results are standard, but we include a
proof for completeness and the convenience of the readers.

Let $\rho\in C_0^{\infty}(\mathbb{R})$ be a real-valued function
with $\supp{\rho}\subset [1/2, 2]$, $\xi, \eta\in\mathbb{R}$,
$\eta\neq 0$, $A>1$ a constant, and
\begin{equation*}
I(\lambda,\epsilon, \xi, \eta)=1_{[-A,
A]}(\xi/\eta)\int_{0}^{\infty} \! \rho(t)e^{i \lambda
\phi_{\epsilon}(t, \xi, \eta)}\,\mathrm{d}t, \quad \lambda>1,
\end{equation*}
where
\begin{equation*}
\phi_{\epsilon}(t, \xi, \eta)=Q_{\epsilon}(t)+(\xi/\eta)t.
\end{equation*}

\begin{lemma}\label{lemma3.1}
Assume that $Q_{\epsilon}\in C^N([1/4, 4])$ ($N\geq 5$) is a
real-valued function such that $|D^j Q_{\epsilon}|\leq C_1$ for
$0\leq j\leq N$ and $|D^2 Q_{\epsilon}|\geq c_1$ for constants $C_1$
and $c_1$. If $\chi\in C_0^{\infty}(\mathbb{R})$ has its support
contained in an interval of length $c_1/12$, then either one of the
following two statements holds.

(1) We have
\begin{equation}
\chi(-\xi/\eta)I(\lambda,\epsilon, \xi, \eta)=O(\lambda^{-(N-1)}).
\label{lemma2.1-1}
\end{equation}

(2) For each pair $(\xi, \eta)$  with $-\xi/\eta\in \supp \chi$,
there exists a unique $t=t(\xi, \eta)\in [1/3, 13/6]$ such that
$t(\xi, \eta)=(Q_{\epsilon}')^{-1}(-\xi/\eta)$ is $(N-1)$-times
differentiable and satisfies
\begin{equation}
D_{t}^1 \phi_{\epsilon}(t(\xi, \eta), \xi, \eta)=0
\label{critical-point-equ0}
\end{equation}
and
\begin{equation}
\begin{split}
&\chi(-\xi/\eta)I(\lambda,\epsilon, \xi,
\eta)=C\chi(-\xi/\eta)1_{[-A,
A]}(\xi/\eta)\rho(t(\xi, \eta))\cdot\\
  &\qquad\quad |D^2_{t}\phi_{\epsilon}(t(\xi, \eta), \xi,
\eta)|^{-1/2}e^{i \lambda \phi_{\epsilon}(t(\xi, \eta), \xi,
\eta)}\lambda^{-1/2}+O(\lambda^{-3/2})
\end{split}\label{lemma2.1-2}
\end{equation}
with $C$ being an absolute constant.

Furthermore, the implicit constants in \eqref{lemma2.1-1} and
\eqref{lemma2.1-2} are independent of $\lambda$, $\epsilon$, $\xi$,
and $\eta$.
\end{lemma}

\begin{proof}
Due to \eqref{origin-condition2}, we observe that $Q_{\epsilon}'$ is
monotone on $[1/4, 4]$ and that, for any $t\in [1/3, 13/6]$ and
$r\in (0, 1/12]$, $Q_{\epsilon}'$ is a bijection from $B(t, r)$
(\footnote{$B(t, r)$ denotes the interval $(t-r, t+r)$.}) to an
interval which contains $B(Q_{\epsilon}'(t), c_1 r)$.

Assume that there exist $a\in [1/2, 2]$ and $(\xi_0, \eta_0)$ with
$-\xi_0/\eta_0\in \supp \chi$ such that $|D_{t}^1 \phi_{\epsilon}(a,
\xi_0, \eta_0)|<c_1/12$, otherwise we get \eqref{lemma2.1-1} by
integration by parts.

Since $D_{t}^1 \phi_{\epsilon}(t, \xi,
\eta)=Q_{\epsilon}'(t)+\xi/\eta$,  we have that $-\xi_0/\eta_0\in
B(Q_{\epsilon}'(a), c_1/12)$. It follows from the observation above
that there exists a unique $a_0\in [1/4, 4]$ such that $a_0\in B(a,
1/12)$ and $Q_{\epsilon}'(a_0)=-\xi_0/\eta_0$. Thus $\supp
\chi\subset B(Q_{\epsilon}'(a_0), c_1/12)$. The observation above
then implies that, for each pair $(\xi, \eta)$ with $-\xi/\eta\in
\supp \chi$, there exists a unique $t(\xi, \eta)\in B(a_0, 1/12)$
such that $Q_{\epsilon}'(t(\xi, \eta))=-\xi/\eta$, which is
\eqref{critical-point-equ0}. In particular, $t(\xi,
\eta)=(Q_{\epsilon}')^{-1}(-\xi/\eta)$, whose differentiability is a
consequence of the inverse function theorem.

Note that $B(t(\xi, \eta), 1/12)\subset [1/4, 9/4]$ and we also have
\begin{equation*}
|D^1_{t} \phi_{\epsilon}(t, \xi, \eta)|=|D^1_{t} \phi_{\epsilon}(t,
\xi, \eta)-D^1_{t} \phi_{\epsilon}(t(\xi, \eta), \xi, \eta)|\geq
c_1|t-t(\xi, \eta)|.
\end{equation*}
Applying to $I(\lambda,\epsilon, \xi, \eta)$ the method of
stationary phase on $B(t(\xi, \eta), 1/12)$ and integration by parts
outside $B(t(\xi, \eta), 1/24)$ yields \eqref{lemma2.1-2}.
\end{proof}

\begin{remark}
A similar argument in high dimensions can be found in the proof of
\cite[Proposition 2.4]{guo}. For the method of stationary phase, the
reader can check \cite[Section 7.7]{hormander}.
\end{remark}

We quote below Li's \cite[Theorem 6.2]{Li2013} with a small
modification in the statement for the sake of our later application,
however its proof remains the same. Let $\sigma\in (0, 1]$,
$\textbf{I}\subset\mathbb{R}$ be a fixed bounded interval, and
$U(\textbf{I})$  a nontrivial subset of $L^2(\textbf{I})$ such that
the $L^2$-norm of every element of $U(\textbf{I})$ is uniformly
bounded by a constant. We say that a function $f\in L^2(\textbf{I})$
is \emph{$\sigma$-uniform in $U(\textbf{I})$} if
\begin{equation*}
\left|\int_{\textbf{I}} \! f(x) \overline{u(x)} \,\textrm{d}x
\right|\leq \sigma \|f\|_{L^2(\textbf{I})} \quad \textrm{for all
$u\in U(\textbf{I})$.}
\end{equation*}

\begin{lemma}\label{Li-Thm6.2}
Let $\mathscr{L}$ be a bounded sublinear functional from
$L^2(\textbf{I})$ to $\mathbb{C}$, $S_\sigma$ the set of all
functions that are $\sigma$-uniform in $U(\textbf{I})$,
\begin{equation*}
A_\sigma = \sup\{|\mathscr{L}(f)|/\|f\|_{L^2(\textbf{I})} : f\in
S_\sigma, f\neq 0\},
\end{equation*}
and
\begin{equation*}
M=\sup_{u\in U(\textbf{I})}|\mathscr{L}(u)|.
\end{equation*}
Then
\begin{equation*}
\|\mathscr{L}\|\leq \max\{A_\sigma,  2\sigma^{-1} M\}.
\end{equation*}
\end{lemma}

We also need the following theorem to handle the minor part in
Section \ref{near-origin}. This theorem is a variant of the results
in \cite[Theorem 2.1]{li2008U} concerning estimates for certain
paraproducts. The only change is that the standard dyadic sequence
$\{2^{\alpha j}\}_{j\in \mathbb Z}$ with $\alpha\in \mathbb
N\backslash \{0\}$ (in \cite{li2008U}) is replaced by a dyadic-like
sequence $\{\Delta_j\}$ here, while the proof remains the same; see
\cite[Section 3 and 4]{li2008U}.

\begin{theorem}\label{para}
Let $L\in \mathbb Z$ and let $\{\Delta_j\}_{j>L}$ be a sequence of
positive numbers which is dyadic-like, i.e. there is a $K\in\mathbb
Z$ s.t. for all $j>L$ and $j+K>L$ the following holds
\begin{align}\label{dl}
{\Delta_{j+K} } \geq 2  {\Delta_j}.
\end{align}
Let $\Phi_1$ be a Schwartz function on $\mathbb{R}$ whose Fourier
transform is a standard bump function supported on
$[-2,-1/2]\cup[1/2,2]$, and $\Phi_2$ be a Schwartz function on
$\mathbb{R}$  such that $\widehat{\Phi_2}$ is a standard bump
function supported on $[-1,1]$ and $\widehat{\Phi_2}(0)=1$. For
$(n_1,n_2)\in \mathbb Z^2$ and $l =1$ or $2$, set
\begin{equation*}
\mathcal M_{l,n_1,n_2}(\xi,\eta) = \sum_{j>L}\widehat{\Phi_l}(\frac
\xi {2^j})  e^{  2\pi i   n_1  \frac \xi   {2^j}   }
\widehat{\Phi_{3-l}}(\frac \eta {\Delta_j} ) e^{  2\pi i   n_2 \frac
\eta {\Delta_j}   }.
\end{equation*}
Then for $l=1$ and $2$, for any $p_1$, $p_2>1$, with $1/r =
1/p_1+1/p_2$, there is a constant $C$ independent of $(n_1,n_2)$
such that for all $f_1\in L^{p_1}(\mathbb R)$, $f_2\in
L^{p_2}(\mathbb R)$, the following holds
\begin{equation*}
\left\|
\Pi_{l,n_1,n_2} (f_1,f_2) \right\|_r \leq
C(1+n_1^2)^{10}(1+n_2^2)^{10} \|f_1\|_{p_1}\|f_2\|_{p_2},
\end{equation*}
where
\begin{equation*}
\Pi_{l,n_1,n_2}(f_1,f_2)(x) =\iint\!
\mathcal{M}_{l,n_1,n_2}(\xi,\eta)
\widehat{f_1}(\xi)\widehat{f_2}(\eta) e^{2\pi i (\xi+\eta)x }
\,\mathrm{d}\xi \,\mathrm{d}\eta.
\end{equation*}
\end{theorem}


%
%
%
%

\section{The main estimates}\label{main-estimates}

Let $\widehat{\varphi}\in C_{0}^{\infty}(\mathbb{R})$ such that
$\widehat{\varphi}=1$ on $\{t\in\mathbb{R} : 3/8\leq|t|\leq 17/8\}$
and $\supp \widehat{\varphi}\subset \{t\in\mathbb{R} :
1/4\leq|t|\leq 9/4\}$. For $j, m\in\mathbb{N}$ denote
$\epsilon_j=2^{-j}$ and
\begin{equation*}
K_{j, m}(\xi, \eta)=\int_{0}^{\infty} \! \rho(t)e^{-2\pi i
2^m\eta\phi_{\epsilon_j}(t, \xi, \eta)}\,\mathrm{d}t,
\end{equation*}
where $\rho$ and $\phi_{\epsilon_j}$ are as defined at the beginning
of Section \ref{phase-function}.

For $f, g\in L^2(\mathbb{R})$ denote, when
$|\gamma'(\epsilon_j)|\leq K_1|\epsilon_j|^{c_1}$,
\begin{equation*}
B_{j, m}^{\varphi}(f, g)(x)=|\gamma'(\epsilon_j)|^{1/2}\iint \!
\widehat{f}(\xi)\widehat{\varphi}(\xi)
\widehat{g}(\eta)\widehat{\varphi}(\eta)  e^{2\pi i
(\gamma'(\epsilon_j)\xi+\eta)x}K_{j, m}(\xi, \eta) \,\mathrm{d}\xi
\,\mathrm{d}\eta,
\end{equation*}
and, when $|\gamma'(\epsilon_j)|\geq K_2|\epsilon_j|^{-c_1}$,
\begin{equation*}
B_{j, m}^{\varphi}(f, g)(x)=|\gamma'(\epsilon_j)|^{-1/2}\iint \!
\widehat{f}(\xi)\widehat{\varphi}(\xi)
\widehat{g}(\eta)\widehat{\varphi}(\eta)  e^{2\pi i
(\xi+\gamma'(\epsilon_j)^{-1}\eta)x}K_{j, m}(\xi, \eta)
\,\mathrm{d}\xi \,\mathrm{d}\eta.
\end{equation*}

\begin{proposition}\label{prop4-1} Assume that
$\Gamma(t)=(t, \gamma(t))\in \textbf{F}(-1, 1)$ (\footnote{We
actually do not need the condition \eqref{mt} for this
proposition.}). For  any $\beta<1$, there exist an $L\in\mathbb{N}$
and a constant $C_{\beta}$  such that whenever $j\geq L$,
$m\in\mathbb{N}$, $n\in\mathbb{Z}$, and $f, g\in L^2(\mathbb{R})$,
we have
\begin{enumerate}
\item if $|\gamma'(\epsilon_j)|\leq K_1|\epsilon_j|^{c_1}$, then
\begin{equation*}
\|B_{j, m}^{\varphi}(f, g)\textbf{1}_{2^m\gamma'(\epsilon_j)^{-1},
n}\|_{1}\leq C_{\beta} C_{j, m}\|f\|_{2} \|g\|_{2},
\end{equation*}
where
\begin{equation}
C_{j, m}=\left\{
\begin{array}{ll}
2^{-m/16} & \textrm{if $|\gamma'(\epsilon_j)|> 2^{-m}$},\\
2^{-\beta m/4}                    & \textrm{if
$|\gamma'(\epsilon_j)|\leq 2^{-m}$};
\end{array}\right. \label{prop4-1-1}
\end{equation}
\item if $|\gamma'(\epsilon_j)|\geq K_2|\epsilon_j|^{-c_1}$, then
\begin{equation*}
\|B_{j, m}^{\varphi}(f, g)\textbf{1}_{2^m\gamma'(\epsilon_j),
n}\|_{1}\leq C_{\beta} C_{j, m}'\|f\|_{2} \|g\|_{2},
\end{equation*}
where
\begin{equation}
C_{j, m}'=\left\{
\begin{array}{ll}
2^{-m/16} & \textrm{if $|\gamma'(\epsilon_j)|<2^{m}$},\\
2^{-\beta m/4}                    & \textrm{if
$|\gamma'(\epsilon_j)|\geq 2^{m}$}.
\end{array}\right. \label{prop4-1-2}
\end{equation}
\end{enumerate}
\end{proposition}

The rest of this section is devoted to the proof of Proposition
\ref{prop4-1}.

We first observe that there is actually a symmetry between the case
$|\gamma'(\epsilon_j)|\leq K_1|\epsilon_j|^{c_1}$ and the case
$|\gamma'(\epsilon_j)|\geq K_2|\epsilon_j|^{-c_1}$, hence we will
only prove the former case while the other one can be handled
similarly. We can also simplify the domain of integration of $B_{j,
m}^{\varphi}(f, g)(x)$ by using a decomposition
$\widehat{\varphi}=\widehat{\varphi}1_{(0,
\infty)}+\widehat{\varphi}1_{(-\infty, 0]}$, which allows us to
restrict the domain to one of the cubes $(\pm [1/4, 9/4])\times(\pm
[1/4, 9/4])$. We will still use $\widehat{\varphi}$ below but with
its support contained in either $[1/4, 9/4]$ or $[-9/4, -1/4]$ (and
this won't cause any problem).

The proof is split into three parts. In the first part, we apply the
$TT^*$ method to estimate $\|B_{j, m}^{\varphi}(f, g)\|_{1}$, during
which procedure we need a standard result from the oscillatory
integral theory and a necessary condition
$|\gamma'(\epsilon_j)|>2^{-m}$. The bound we get (see
\eqref{4-1-estimate2} below) is efficient when
$|\gamma'(\epsilon_j)|$ is large but inefficient when
$|\gamma'(\epsilon_j)|$ is close to $2^{-m}$. In the second part,
with the help of Lemma \ref{Li-Thm6.2} (the method of
$\sigma$-uniformity introduced in Li~\cite{Li2013}), we can put
certain restrictions on the function $f$ (or $g$) and reduce the
estimate of $\|B_{j, m}^{\varphi}(f, g)\|_{1}$ to a restricted
version, to which the $TT^*$ method can be applied without extra
assumptions on the size of $|\gamma'(\epsilon_j)|$. The bound we get
in this part (see \eqref{estimate4-2-1} and \eqref{estimate4-2-3}
below) is efficient when $|\gamma'(\epsilon_j)|$ is small (even when
$|\gamma'(\epsilon_j)|$ is close to $2^{-m}$) but inefficient when
$|\gamma'(\epsilon_j)|$ is large (see also
Lie~\cite[P.~18]{Lie2011}). In the last part, we take advantage of
both results and prove the desired estimate.


\subsection{Part 1: $j\geq L$, $m\in \mathbb{N}$ such that $|\gamma'(\epsilon_j)|>2^{-m}$}\label{case1}

We first prove that, for $h\in L^2(\mathbb{R})$,
\begin{equation}
\left|\int \! B_{j, m}^{\varphi}(f, g)(x)h(x) \,
\mathrm{d}x\right|\leq
C(2^m|\gamma'(\epsilon_j)|)^{-1/6}\|f\|_2\|g\|_2\cdot(2^{-m}|\gamma'(\epsilon_j)|)^{1/2}\|h\|_2,\label{4-1-estimate1}
\end{equation}
which trivially leads to the estimate
\begin{equation}
\|B_{j, m}^{\varphi}(f, g)\textbf{1}_{2^m\gamma'(\epsilon_j)^{-1},
n}\|_1 \leq C(2^m|\gamma'(\epsilon_j)|)^{-1/6}\|f\|_2\|g\|_2.
\label{4-1-estimate2}
\end{equation}

We can find a finite open cover of the interval $[-10, 10]$ by using
open intervals of length $c_1/12$, associated to which we can
construct a partition of unity. By inserting this partition of unity
we reduce the estimate of $\int \! B_{j, m}^{\varphi}(f, g)(x)h(x)
\, \mathrm{d}x$ to
\begin{align*}
&\int \! \widetilde{B}_{j, m}^{\varphi}(f, g)(x)h(x) \, \mathrm{d}x=  \\
&\quad |\gamma'(\epsilon_j)|^{1/2} \iint \!
\widehat{f}\widehat{\varphi}(\xi) \widehat{g}\widehat{\varphi}(\eta)
\mathcal{F}^{-1}[h](\gamma'(\epsilon_j)\xi+\eta)
\chi(-\xi/\eta)K_{j, m}(\xi, \eta) \,\mathrm{d}\xi \,\mathrm{d}\eta,
\end{align*}
where $\chi$ is smooth and supported in an interval of length
$c_1/12$.

We can then apply Lemma \ref{lemma3.1} to $\chi(-\xi/\eta)K_{j,
m}(\xi, \eta)$. If \eqref{lemma2.1-1} holds, then an application of
H\"older's inequality yields
\begin{equation}
\left|\int \! \widetilde{B}_{j, m}^{\varphi}(f, g)(x)h(x) \,
\mathrm{d}x\right|\leq C
2^{-m/2}\|f\|_2\|g\|_2\cdot(2^{-m}|\gamma'(\epsilon_j)|)^{1/2}\|h\|_2.\label{estimate4-1}
\end{equation}
This estimate immediately leads to \eqref{4-1-estimate1}.

Below we will assume that the second statement in Lemma
\ref{lemma3.1} holds. Applying \eqref{lemma2.1-2} yields
\begin{align}
&\int \! \widetilde{B}_{j, m}^{\varphi}(f, g)(x)h(x) \, \mathrm{d}x=C(2^{-m}|\gamma'(\epsilon_j)|)^{1/2}\cdot  \nonumber \\
&\quad \iint \! \widehat{f}\widehat{\varphi}(\xi)
\widehat{g}\widehat{\varphi}(\eta)
\mathcal{F}^{-1}[h](\gamma'(\epsilon_j)\xi+\eta)a(\xi, \eta)
e^{-2\pi i 2^m\eta\phi_{\epsilon_j}(t(\xi, \eta), \xi, \eta)}
\,\mathrm{d}\xi \,\mathrm{d}\eta, \label{4-1-double-int}
\end{align}
where we have omitted the error term in \eqref{lemma2.1-2} (since it
leads to the same bound as in \eqref{estimate4-1}), and $a(\xi,
\eta)$ is defined as
\begin{equation*}
a(\xi, \eta)=\chi(-\xi/\eta)\rho(t(\xi,
\eta))|\eta|^{-1/2}|D^2_{t}\phi_{\epsilon_j}(t(\xi, \eta), \xi,
\eta)|^{-1/2}.
\end{equation*}
Applying to the double integral in \eqref{4-1-double-int} a change
of variables $\gamma'(\epsilon_j)\xi+\eta\rightarrow \xi$,
$\eta/\gamma'(\epsilon_j)\rightarrow \eta$ and then H\"{o}lder's
inequality yields
\begin{equation}
\left|\int \! \widetilde{B}_{j, m}^{\varphi}(f, g)(x)h(x) \,
\mathrm{d}x\right|\leq C\|T_{j, m}(f,
g)\|_2\cdot(2^{-m}|\gamma'(\epsilon_j)|)^{1/2}\|h\|_2,\label{4-1-after-holder}
\end{equation}
where
\begin{align*}
T_{j, m}(f, g)(\xi)&=\int \!
\widehat{f}\widehat{\varphi}(\gamma'(\epsilon_j)^{-1}\xi-\eta)
\widehat{g}\widehat{\varphi}(\gamma'(\epsilon_j)\eta)
a(\gamma'(\epsilon_j)^{-1}\xi-\eta, \gamma'(\epsilon_j)\eta)\cdot \\
&\quad \quad e^{-2\pi i
2^m(\gamma'(\epsilon_j)\eta)\phi_{\epsilon_j}(t(\gamma'(\epsilon_j)^{-1}\xi-\eta,
\gamma'(\epsilon_j)\eta), \gamma'(\epsilon_j)^{-1}\xi-\eta,
\gamma'(\epsilon_j)\eta)} \,\mathrm{d}\eta.
\end{align*}

We then have, after a change of variables,
\begin{align}
\|T_{j, m}(f, g)\|_2^2&=\int \! T_{j, m}(f, g)(\xi)\overline{T_{j,
m}(f, g)(\xi)} \,\mathrm{d}\xi \nonumber\\
   &=\int \mathrm{d}\tau \iint \! F_{\tau}(x)G_{\tau}(y)A_{\tau}(x, y)e^{-2\pi i
2^m P_{\tau}(x, y)}\,\mathrm{d}x \,\mathrm{d}y,\label{4-1-inner-int}
\end{align}
where
\begin{equation*}
F_{\tau}(x)=\widehat{f}\widehat{\varphi}(x-\tau)\overline{\widehat{f}\widehat{\varphi}(x)},
\end{equation*}
\begin{equation*}
G_{\tau}(y)=\widehat{g}\widehat{\varphi}(y+\gamma'(\epsilon_j)\tau)\overline{\widehat{g}\widehat{\varphi}(y)},
\end{equation*}
\begin{equation*}
A_{\tau}(x, y)=a(x-\tau, y+\gamma'(\epsilon_j)\tau)\overline{a(x,
y)},
\end{equation*}
and
\begin{equation*}
P_{\tau}(x, y)=P_1(x-\tau, y+\gamma'(\epsilon_j)\tau)-P_1(x, y)
\end{equation*}
with
\begin{equation*}
P_1(x, y)=y\phi_{\epsilon_j}(t(x, y), x, y).
\end{equation*}

In order to estimate the inner double integral in
\eqref{4-1-inner-int}, we first show that there exists an
$L\in\mathbb{N}$ such that if $j\geq L$ then
\begin{equation}
\left|\frac{\partial^2 P_{\tau}}{\partial y \partial x}(x,
y)\right|\asymp |\tau|.\label{4-1-phase}
\end{equation}

Recall that $t(x, y)$ satisfies \eqref{critical-point-equ0} (with
$\xi$, $\eta$, and $\epsilon$ replaced by $x$, $y$, and $\epsilon_j$
respectively). By implicit differentiation, we get
\begin{equation*}
\frac{\partial t}{\partial x}(x, y)=-\frac{1}{yQ_{\epsilon_j}''(t(x,
y))} \quad \textrm{and}\quad \frac{\partial t}{\partial y}(x,
y)=-\frac{Q_{\epsilon_j}'(t(x, y))}{yQ_{\epsilon_j}''(t(x, y))}.
\end{equation*}
By \eqref{origin-condition1}, \eqref{origin-condition2}, and
\eqref{origin-condition3}, we then have
\begin{equation*}
\frac{\partial^2 t}{\partial x \partial y}(x,
y)=\frac{1}{y^2}\frac{(Q_{\epsilon_j}'')^2-Q_{\epsilon_j}'Q_{\epsilon_j}'''}{(Q_{\epsilon_j}'')^3}(t(x,
y))\asymp 1
\end{equation*}
and
\begin{equation*}
\frac{\partial^2 t}{\partial y^2}(x,
y)=\frac{1}{y^2}\frac{Q_{\epsilon_j}'(2(Q_{\epsilon_j}'')^2-Q_{\epsilon_j}'Q_{\epsilon_j}''')}{(Q_{\epsilon_j}'')^3}(t(x,
y))\lesssim 1.
\end{equation*}
By using \eqref{critical-point-equ0} we also get
\begin{equation*}
\frac{\partial^2 P_1}{\partial y \partial x}(x, y)=\frac{\partial
t}{\partial y}(x, y).
\end{equation*}
Noticing that $|\gamma'(\epsilon_j)|$ is small if $L$ is large, by
the mean value theorem we get \eqref{4-1-phase}.

Let $\tau_0=(2^m|\gamma'(\epsilon_j)|)^{-1/3}$. We have the
following splitting of \eqref{4-1-inner-int}:
\begin{align*}
&\|T_{j, m}(f, g)\|_2^2=\\
&\quad \left(\int_{|\tau|<\tau_0}+\int_{\tau_0\leq |\tau|\leq
10}\mathrm{d}\tau\right) \cdot\iint \!
F_{\tau}(x)G_{\tau}(y)A_{\tau}(x, y)e^{-2\pi i 2^m P_{\tau}(x,
y)}\,\mathrm{d}x \,\mathrm{d}y.
\end{align*}
Applying the trivial estimate and H{\"o}rmander's \cite[Theorem
1.1]{hormander1973} to the two parts above (and also H\"{o}lder's
inequality) yields
\begin{align*}
\|T_{j, m}(f, g)\|_2^2&\leq
C\tau_0\|f\|_2^2\|g\|_2^2+C\int_{\tau_0\leq |\tau|\leq 10}  \!
(2^m|\tau|)^{-1/2}
\|F_{\tau}\|_2\|G_{\tau}\|_2 \,\mathrm{d}\tau  \\
 &\leq C(\tau_0+(2^m|\gamma'(\epsilon_j)|\tau_0)^{-1/2})\|f\|_2^2\|g\|_2^2\\
 &\leq C(2^m|\gamma'(\epsilon_j)|)^{-1/3}\|f\|_2^2\|g\|_2^2.
\end{align*}

To conclude, the desired estimate \eqref{4-1-estimate1} follows from
\eqref{estimate4-1}, \eqref{4-1-after-holder}, and the estimate
above of $\|T_{j, m}(f, g)\|_2$.


\subsection{Part 2: $j\geq L$, $m\in \mathbb{N}$}\label{case2}

We can find a finite open cover of the interval $[-36C_1, 36C_1]$ by
using open intervals of length $c_1/24$, associated to which we can
construct a partition of unity $\{\chi_s : 1\leq s\leq \Theta\}$
such that  $\sum_s \chi_s\equiv 1$ in $[-36C_1, 36C_1]$ and each
$\chi_s$ is smooth and supported in an interval that belongs to the
finite open cover above.

Lemma \ref{lemma3.1} will be applied to $\chi_s(-\xi/\eta)K_{j,
m}(\xi, \eta)$ (below). Here we denote $S$ to be the collection of
all $1\leq s\leq \Theta$ for which the second statement in Lemma
\ref{lemma3.1} holds. Let $\textbf{I}$ be either $[1/4, 9/4]$ or
$[-9/4, -1/4]$, and
\begin{equation*}
U(\textbf{I}):=\{u_{s, r, \eta}(\xi)\in L^2(\textbf{I}) : s\in S,
r\in\mathbb{R}, 1/16C_1\leq |\eta|\leq 9C_1\},
\end{equation*}
where
\begin{equation*}
u_{s, r, \eta}(\xi)=\chi_s(-\xi/\eta)e^{2\pi i(
2^m\eta\phi_{\epsilon_j}(t(\xi, \eta), \xi, \eta)+r\xi)}.
\end{equation*}

According to Lemma \ref{Li-Thm6.2}, we will finish this part in
three steps.

\noindent{}\textbf{Step 1}: Let $\widehat{f}|_{\textbf{I}}$, the
restriction of $\widehat{f}$ to $\textbf{I}$, be an arbitrary
function in $L^2(\textbf{I})$ that is $\sigma$-uniform in
$U(\textbf{I})$.

We first note that $B_{j, m}^{\varphi}(f, g)(x)$ in the time space
can be expressed as
\begin{equation} B_{j, m}^{\varphi}(f,
g)(x)=|\gamma'(\epsilon_j)|^{1/2} \int_{0}^{\infty} \!
f*\varphi(\gamma'(\epsilon_j)x-2^m t)g*\varphi(x-2^m
Q_{\epsilon_j}(t))\rho(t) \,\mathrm{d}t, \label{time-space-formula}
\end{equation}
which leads to, for $h\in L^2(\mathbb{R})$,
\begin{align*}
&\int \! B_{j, m}^{\varphi}(f, g)(x)h(x) \,
\mathrm{d}x=|\gamma'(\epsilon_j)|^{1/2}\cdot\\
&\quad \sum_{l\in\mathbb{Z}}\iint_{0}^{\infty} \!
f*\varphi(\gamma'(\epsilon_j)x-2^m t)g_{j, m, l}(x-2^m
Q_{\epsilon_j}(t))\rho(t)(\textbf{1}_{|\gamma'(\epsilon_j)|^{-1},
l}h)(x) \,\mathrm{d}t\,\mathrm{d}x,
\end{align*}
where $g_{j, m, l}=1_{I_{j, m, l}}\cdot g*\varphi$ with $I_{j, m,
l}=[\alpha_{j, l}-C_1 2^m, \alpha_{j, l+1}+C_1 2^m]$ and $\alpha_{j,
l}=|\gamma'(\epsilon_j)|^{-1}l$. In the frequency space we then have
\begin{align*}
&\int \! B_{j, m}^{\varphi}(f, g)(x)h(x) \,
\mathrm{d}x=|\gamma'(\epsilon_j)|^{1/2}\cdot\\
&\quad\sum_{l\in\mathbb{Z}}\iiint \!
\widehat{f}\widehat{\varphi}(\xi)e^{2\pi i \gamma'(\epsilon_j)\xi
x}\widehat{g_{j, m, l}}(\eta)e^{2\pi i \eta x}K_{j, m}(\xi,
\eta)(\textbf{1}_{|\gamma'(\epsilon_j)|^{-1}, l}h)(x)
\,\mathrm{d}x\,\mathrm{d}\xi \,\mathrm{d}\eta.
\end{align*}

Let $\widehat{\varphi_1}\in C_{0}^{\infty}(\mathbb{R})$ such that
$\widehat{\varphi_1}(\eta)=1$ if $|\eta|\in [1/8C_1, 9C_1/2]$ and
$\supp \widehat{\varphi_1}\subset \{x\in \mathbb{R} : 1/16C_1\leq
|x|\leq 9C_1\}$. By using
$1=\widehat{\varphi_1}(\eta)+(1-\widehat{\varphi_1}(\eta))$ and the
power series of $e^{2\pi i \gamma'(\epsilon_j)\xi (x-\alpha_{j,
l})}$, we get

\begin{equation*}
\int \! B_{j, m}^{\varphi}(f, g)(x)h(x) \, \mathrm{d}x=I+II,
\end{equation*}
where
\begin{align*}
&I=|\gamma'(\epsilon_j)|^{1/2}\sum_{l\in\mathbb{Z}}\sum_{p=0}^{\infty}\frac{(2\pi
i)^p}{p!}\iint \! \widehat{f}\widehat{\varphi}(\xi)\xi^p e^{2\pi i
\gamma'(\epsilon_j)\alpha_{j, l}\xi}\cdot\\
&\quad\widehat{g_{j, m, l}}(\eta)\widehat{\varphi_1}(\eta) K_{j,
m}(\xi, \eta)\mathcal{F}^{-1}[(\gamma'(\epsilon_j)(\cdot-\alpha_{j,
l}))^p(\textbf{1}_{|\gamma'(\epsilon_j)|^{-1}, l}h)(\cdot)](\eta)
\,\mathrm{d}\xi \,\mathrm{d}\eta
\end{align*}
and
\begin{align*}
&II=|\gamma'(\epsilon_j)|^{1/2}\sum_{l\in\mathbb{Z}}\sum_{p=0}^{\infty}\frac{(2\pi
i)^p}{p!}\iint \! \widehat{f}\widehat{\varphi}(\xi)\xi^p e^{2\pi i
\gamma'(\epsilon_j)\alpha_{j, l}\xi}\cdot\\
&\quad\widehat{g_{j, m, l}}(\eta)(1-\widehat{\varphi_1}(\eta)) K_{j,
m}(\xi, \eta)\mathcal{F}^{-1}[(\gamma'(\epsilon_j)(\cdot-\alpha_{j,
l}))^p(\textbf{1}_{|\gamma'(\epsilon_j)|^{-1}, l}h)(\cdot)](\eta)
\,\mathrm{d}\xi \,\mathrm{d}\eta.
\end{align*}

We first estimate Sum $II$. When $1-\widehat{\varphi_1}(\eta)\neq
0$, Remark \ref{remark2-1} (2) implies that the gradient of the
phase function of $K_{j, m}(\xi, \eta)$ has a uniform lower bound,
which leads to the bound $K_{j, m}(\xi, \eta)=O(2^{-m})$. Then by
H\"{o}lder's inequality we get
\begin{equation*}
|II|\leq
C2^{-m}|\gamma'(\epsilon_j)|^{1/2}\|1_{\textbf{I}}\widehat{f}\|_2\sum_{l\in\mathbb{Z}}\|g_{j,
m, l}\|_2\|\textbf{1}_{|\gamma'(\epsilon_j)|^{-1}, l}h\|_2.
\end{equation*}
Applying the Cauchy-Schwarz inequality yields
\begin{equation}
    |II|\leq \left\{\begin{array}{ll}
    C
2^{-m/2}\|1_{\textbf{I}}\widehat{f}\|_2\|g\|_2\cdot(2^{-m}|\gamma'(\epsilon_j)|)^{1/2}\|h\|_2,
& \textrm{if $|\gamma'(\epsilon_j)|\leq
    2^{-m}$},\\
    C|\gamma'(\epsilon_j)|^{1/2}\|1_{\textbf{I}}\widehat{f}\|_2\|g\|_2\cdot(2^{-m}|\gamma'(\epsilon_j)|)^{1/2}\|h\|_2, & \textrm{if $|\gamma'(\epsilon_j)|> 2^{-m}$}.
    \end{array}\right.\label{4-2-error-1}
\end{equation}

The estimate of Sum $I$, by using the partition of unity we have
constructed at the beginning of this subsection, can be reduced to
\begin{align*}
&I_s=|\gamma'(\epsilon_j)|^{1/2}\sum_{l\in\mathbb{Z}}\sum_{p=0}^{\infty}\frac{(2\pi
i)^p}{p!}\iint \! \widehat{f}\widehat{\varphi}(\xi)\xi^p e^{2\pi i
\gamma'(\epsilon_j)\alpha_{j, l}\xi}\widehat{g_{j, m,
l}}(\eta)\widehat{\varphi_1}(\eta)\cdot\\
&\quad\chi_s(-\xi/\eta) K_{j, m}(\xi,
\eta)\mathcal{F}^{-1}[(\gamma'(\epsilon_j)(\cdot-\alpha_{j,
l}))^p(\textbf{1}_{|\gamma'(\epsilon_j)|^{-1}, l}h)(\cdot)](\eta)
\,\mathrm{d}\xi \,\mathrm{d}\eta
\end{align*}
for any $1\leq s\leq \Theta$. We apply Lemma \ref{lemma3.1} to
$\chi_s(-\xi/\eta)K_{j, m}(\xi, \eta)$. If \eqref{lemma2.1-1} holds,
then $I_s$ is bounded by \eqref{4-2-error-1} too. Hence we may
assume that the second statement in Lemma \ref{lemma3.1} holds.
Applying \eqref{lemma2.1-2} yields
\begin{align*}
I_s=&C(2^{-m}|\gamma'(\epsilon_j)|)^{1/2}\sum_{l\in\mathbb{Z}}\sum_{p=0}^{\infty}\frac{(2\pi i)^p}{p!}\cdot \\
     &\int \! \mathfrak{M}(\eta) \widehat{g_{j, m, l}}(\eta)
\mathcal{F}^{-1}[(\gamma'(\epsilon_j)(\cdot-\alpha_{j,
l}))^p(\textbf{1}_{|\gamma'(\epsilon_j)|^{-1},
l}h)(\cdot)](\eta)\,\mathrm{d}\eta,
\end{align*}
where we have omitted the error term in \eqref{lemma2.1-2} (since it
leads to the same bound as in \eqref{4-2-error-1}), and
$\mathfrak{M}(\eta)$ is defined as
\begin{equation*}
\mathfrak{M}(\eta):=\int_{\textbf{I}} \!  b(\xi,
\eta)\widehat{f}(\xi)\chi_s(-\xi/\eta)e^{-2\pi i (
2^m\eta\phi_{\epsilon_j}(t(\xi, \eta), \xi,
\eta)-\gamma'(\epsilon_j)\alpha_{j, l}\xi)} \,\mathrm{d}\xi
\end{equation*}
with
\begin{equation*}
b(\xi,
\eta)=\widehat{\varphi}(\xi)\widehat{\varphi_1}(\eta)\xi^p\rho(t(\xi,
\eta))|\eta|^{-1/2}|D^2_{t}\phi_{\epsilon_j}(t(\xi, \eta), \xi,
\eta)|^{-1/2}.
\end{equation*}
Using the Fourier series of $b(\xi, \eta)$ and the assumption that
$\widehat{f}|_{\textbf{I}}$ is $\sigma$-uniform in $U(\textbf{I})$,
we have
\begin{equation*}
|\mathfrak{M}(\eta)|\leq
C9^{p}\sigma\|1_{\textbf{I}}\widehat{f}\|_2.
\end{equation*}
Hence by using H\"{o}lder's and the Cauchy-Schwarz inequalities we
get
\begin{equation*}
   |I_s|\leq \left\{\begin{array}{ll}
    C\sigma\|1_{\textbf{I}}\widehat{f}\|_2\|g\|_2\cdot (2^{-m}|\gamma'(\epsilon_j)|)^{1/2}\|h\|_2, & \textrm{if $|\gamma'(\epsilon_j)|\leq
    2^{-m}$},\\
    C(2^m|\gamma'(\epsilon_j)|)^{1/2}\sigma\|1_{\textbf{I}}\widehat{f}\|_2\|g\|_2\cdot (2^{-m}|\gamma'(\epsilon_j)|)^{1/2}\|h\|_2, & \textrm{if $|\gamma'(\epsilon_j)|> 2^{-m}$}.
    \end{array}\right.
\end{equation*}

To conclude Step 1, if $\sigma>2^{-m/2}$, then the bound above of
$I_s$ and \eqref{4-2-error-1} lead to, for $h\in
L^{\infty}(\mathbb{R})$,
\begin{align}
&\left|\int \! B_{j, m}^{\varphi}(f,
g)(x)\textbf{1}_{2^m\gamma'(\epsilon_j)^{-1}, n}(x)h(x) \,
\mathrm{d}x\right| \nonumber\\
&\quad \leq \left\{\begin{array}{ll}
    C\sigma\|1_{\textbf{I}}\widehat{f}\|_2\|g\|_2\|h\|_{\infty}, & \textrm{if $|\gamma'(\epsilon_j)|\leq
    2^{-m}$},\\
    C(2^m|\gamma'(\epsilon_j)|)^{1/2}\sigma\|1_{\textbf{I}}\widehat{f}\|_2\|g\|_2\|h\|_{\infty}, & \textrm{if $|\gamma'(\epsilon_j)|> 2^{-m}$}.
    \end{array}\right.\label{step1-bound}
\end{align}


\noindent{}\textbf{Step 2}: We now  assume that
$\widehat{f}|_{\textbf{I}}\in U(\textbf{I})$.

By using \eqref{time-space-formula}, a change of variables
$x\rightarrow 2^m
\gamma'(\epsilon_j)^{-1}(x+\gamma'(\epsilon_j)Q_{\epsilon_j}(t))$,
and H\"older's inequality, we have, for $h\in
L^{\infty}(\mathbb{R})$,
\begin{align}
\left|\int \! B_{j, m}^{\varphi}(f, g)(x)h(x) \,
\mathrm{d}x\right|&=2^m|\gamma'(\epsilon_j)|^{-1/2}\bigg|\iint_{0}^{\infty}
\!
f*\varphi(2^m(x+\gamma'(\epsilon_j)Q_{\epsilon_j}(t)-t))\cdot \nonumber\\
&\quad\quad g*\varphi(2^m \gamma'(\epsilon_j)^{-1}x)h_{j,
m}(x+\gamma'(\epsilon_j)Q_{\epsilon_j}(t))\rho(t) \,\mathrm{d}t\,\mathrm{d}x\bigg|\nonumber\\
&\leq C\|g\|_2\|T_1(h)\|_2,\label{4-2-bound1}
\end{align}
where $h_{j, m}(x)=h(2^m \gamma'(\epsilon_j)^{-1}x)$ and
\begin{equation*}
T_1(h)(x)=2^{m/2}\int_{0}^{\infty} \!
f*\varphi(2^m(x+\gamma'(\epsilon_j)Q_{\epsilon_j}(t)-t))h_{j,
m}(x+\gamma'(\epsilon_j)Q_{\epsilon_j}(t))\rho(t) \,\mathrm{d}t.
\end{equation*}

Let $\widehat{f}|_{\textbf{I}}=u_{s, r, \eta}(\xi)$ for arbitrarily
fixed $s\in S$, $r\in\mathbb{R}$, and $1/16C_1\leq |\eta|\leq 9C_1$.
By applying the Fourier inversion formula to $f*\varphi$ and
changing variables, we get
\begin{equation}
\|T_1(h)\|_2^2=2^m\int \! \left|\int_0^{\infty}K_{1}(x, t) h_{j,
m}(x-2^{-m}r+\gamma'(\epsilon_j)Q_{\epsilon_j}(t))\rho(t)
\,\mathrm{d}t\right|^2\,\mathrm{d}x, \label{4-2-T-1}
\end{equation}
where
\begin{equation}
K_{1}(x, t)=\int \! \widehat{\varphi}(\xi)\chi_s(-\xi/\eta)e^{2\pi i
2^m\eta[\phi_{\epsilon_j}(t(\xi, \eta), \xi, \eta)+y(x,
t)(\xi/\eta)]}\, \textrm{d} \xi \label{4-2-SP}
\end{equation}
with
\begin{equation*}
y=y(x, t)=x+\gamma'(\epsilon_j)Q_{\epsilon_j}(t)-t.
\end{equation*}

Let $\chi_M$ be a smooth cut-off function supported in $[-M, M]$,
which equals $1$ in $[-M/2, M/2]$. We decompose the right hand side
of \eqref{4-2-T-1} into two parts by using the decomposition
$1=(1-\chi_M(x))+\chi_M(x)$ to restrict the integration domain of
$x$ to $\{x\in \mathbb{R} : |x|\geq M/2\}$ and $\{x\in \mathbb{R} :
|x|<M\}$ respectively for a sufficiently large constant $M$. The
former part is bounded by
\begin{equation}
C2^{-m}\|h\|_{\infty}^2,\label{minkowski-estimate}
\end{equation}
since integration by parts yields $K_{1}(x, t)=O(2^{-m}|x|^{-1})$.

We next consider the latter part with $|x|<M$. After inserting a
partition of unity, we may replace $K_{1}(x, t)$ by
$\widetilde{\chi}(-y(x, t))K_{1}(x, t)$ with a smooth cut-off
function $\widetilde{\chi}$ supported in an interval of sufficiently
small length. Then by repeating the argument in the proof of Lemma
\ref{lemma3.1}, we have that either $\widetilde{\chi}(-y(x,
t))K_{1}(x, t)=O(2^{-m})$ (leading to the bound
\eqref{minkowski-estimate}) or the phase function in \eqref{4-2-SP}
has a critical point $\xi(x, t)$ satisfying
\begin{equation*}
t(\xi(x, t), \eta)=-y(x, t).
\end{equation*}
This equation, together with $\partial_{t}\phi_{\epsilon_j}(t(\xi,
\eta), \xi, \eta)=0$ (namely, the equation
\eqref{critical-point-equ0} satisfied by $t(\xi, \eta)$), yields
\begin{equation*}
\xi(x, t)=-\eta Q_{\epsilon_j}'(-y(x, t)).
\end{equation*}
By using the method of stationary phase in a neighborhood of $\xi(x,
t)$ and integration by parts away from it, we get the following
asymptotic formula.
\begin{equation*}
\begin{split}
\widetilde{\chi}(-y(x, t))K_{1}(x, t)&=C\widetilde{\chi}(-y(x,
t))\chi_s(-\xi(x, t)/\eta)\widehat{\varphi}(\xi(x,
t))|\partial_{\xi}t(\xi(x, t),
\eta)|^{-1/2}\cdot \\
&\quad e^{2\pi i 2^m\eta Q_{\epsilon_j}(-y(x,
t))}2^{-m/2}+O(2^{-3m/2}).
\end{split}
\end{equation*}
By using the leading term above and a change of variables
$u=Q_{\epsilon_j}(t)$, we now need to estimate
\begin{equation}
\int \! \chi_M(x)\left|\int h_{j,
m}(x-2^{-m}r+\gamma'(\epsilon_j)u)k(x, u)e^{2\pi i 2^m\eta
Q_{\epsilon_j}(-y(x, Q_{\epsilon_j}^{-1}(u)))}
\,\mathrm{d}u\right|^2 \,\mathrm{d}x,         \label{T1-part2}
\end{equation}
where
\begin{equation*}
\begin{split}
k(x, u)&=\widetilde{\chi}(-y(x,
Q_{\epsilon_j}^{-1}(u)))\chi_s(-\xi(x,
Q_{\epsilon_j}^{-1}(u))/\eta)\widehat{\varphi}(\xi(x,
Q_{\epsilon_j}^{-1}(u)))\cdot \\
&\quad |\partial_{\xi}t(\xi(x, Q_{\epsilon_j}^{-1}(u)),
\eta)|^{-1/2}\rho(Q_{\epsilon_j}^{-1}(u))(Q_{\epsilon_j}'(Q_{\epsilon_j}^{-1}(u)))^{-1}.
\end{split}
\end{equation*}

We use the $TT^*$ method for $\eqref{T1-part2}$. By changing
variables $u_1=\upsilon+\tau$, $u_2=\upsilon$, followed by
$x\rightarrow x-\gamma'(\epsilon_j)\upsilon$, \eqref{T1-part2}
becomes
\begin{equation}
\int \textrm{d}\tau \int \! H_{\tau}(x) \,\mathrm{d}x \int \!
K_{\tau, x}(\upsilon)e^{2\pi i 2^m\eta P_{\tau, x}(\upsilon)}
\,\mathrm{d}\upsilon, \label{def-T1}
\end{equation}
where all three integrals are over some finite intervals,
\begin{equation*}
H_{\tau}(x)=h_{j,
m}(x-2^{-m}r+\gamma'(\epsilon_j)\tau)\overline{h_{j, m}(x-2^{-m}r)},
\end{equation*}
\begin{equation*}
K_{\tau,
x}(\upsilon)=\chi_M(x-\gamma'(\epsilon_j)\upsilon)k(x-\gamma'(\epsilon_j)\upsilon,
\upsilon+\tau)\overline{k(x-\gamma'(\epsilon_j)\upsilon, \upsilon)},
\end{equation*}
and
\begin{equation*}
P_{\tau, x}(\upsilon)=P_{2}(x+\gamma'(\epsilon_j)\tau,
\upsilon+\tau)-P_{2}(x, \upsilon)
\end{equation*}
with
\begin{equation*}
P_{2}(x,
\upsilon)=Q_{\epsilon_j}(-[x-Q_{\epsilon_j}^{-1}(\upsilon)]).
\end{equation*}

Before applying integration by parts to the innermost integral in
\eqref{def-T1} we first estimate its phase function $P_{\tau,
x}(\upsilon)$. Actually we have that if $|\gamma'(\epsilon_j)|/|x|$
is sufficiently small then
\begin{equation}
|D_{\upsilon}P_{\tau, x}(\upsilon)|\asymp |x||\tau|
\label{claim4-2-1}
\end{equation}
and
\begin{equation}
|D^2_{\upsilon}P_{\tau, x}(\upsilon)|\lesssim |x||\tau|.
\label{claim4-2-2}
\end{equation}
The \eqref{claim4-2-1} follows from the mean value theorem and the
following estimates
\begin{equation*}
\frac{\partial^2 P_{2}}{\partial x\partial \upsilon}(x,
\upsilon)=-\frac{Q_{\epsilon_j}^{''}(-[x-Q_{\epsilon_j}^{-1}(\upsilon)])}{Q_{\epsilon_j}^{'}(Q_{\epsilon_j}^{-1}(\upsilon))}
\asymp 1
\end{equation*}
and
\begin{equation*}
\frac{\partial^2 P_{2}}{\partial \upsilon^2}(x,
\upsilon)=x\cdot\frac{Q_{\epsilon_j}^{'}(-[x-Q_{\epsilon_j}^{-1}(\upsilon)])}{(Q_{\epsilon_j}^{'}(Q_{\epsilon_j}^{-1}(\upsilon)))^2}\cdot
\frac{(Q_{\epsilon_j}'')^2-Q_{\epsilon_j}'Q_{\epsilon_j}'''}{(Q_{\epsilon_j}')^2}(c)\asymp
|x|,
\end{equation*}
where $c$ is between $-[x-Q_{\epsilon_j}^{-1}(\upsilon)]$ and
$Q_{\epsilon_j}^{-1}(\upsilon)$. The \eqref{claim4-2-2} can be
proved similarly.

Therefore, if $|\gamma'(\epsilon_j)|/|x|$ is sufficiently small, for
any $\beta<1$ we have
\begin{equation*}
\left|\int \! K_{\tau, x}(\upsilon)e^{2\pi i 2^m\eta P_{\tau,
x}(\upsilon)} \,\mathrm{d}\upsilon\right|\leq C\min\{1,
(2^m|x||\tau|)^{-1}\}\leq C (2^m|x||\tau|)^{-\beta}.
\end{equation*}
We now estimate \eqref{def-T1} by splitting it into two parts
(depending on the size of $|\gamma'(\epsilon_j)|/|x|$) and using the
trivial estimate and the bound above respectively. Then it is
bounded by
\begin{align}
C(|\gamma'(\epsilon_j)|+2^{-\beta m} )\|h\|_{\infty}^2.
\label{4-2-bound2}
\end{align}

To conclude Step 2,  by \eqref{4-2-bound1}, \eqref{4-2-T-1},
\eqref{minkowski-estimate}, and \eqref{4-2-bound2}, we get, for
$h\in L^{\infty}(\mathbb{R})$,
\begin{align}
&\left|\int \! B_{j, m}^{\varphi}(f,
g)(x)\textbf{1}_{2^m\gamma'(\epsilon_j)^{-1}, n}(x)h(x) \,
\mathrm{d}x\right|   \nonumber \\
&\quad \leq \left\{\begin{array}{ll}
    C2^{-\beta m/2}\|g\|_2\|h\|_{\infty}, & \textrm{if $|\gamma'(\epsilon_j)|\leq
    2^{-m}$},\\
    C(\max\{|\gamma'(\epsilon_j)|, 2^{-\beta m}\})^{1/2}
\|g\|_2\|h\|_{\infty}, & \textrm{if $|\gamma'(\epsilon_j)|>
2^{-m}$}.
    \end{array}\right.\label{step2-bound}
\end{align}

\noindent{}\textbf{Step 3}: To conclude this subsection (namely,
Part 2), by using Lemma \ref{Li-Thm6.2} and the estimates
\eqref{step1-bound} and \eqref{step2-bound}, we get that for any
$\beta<1$
\begin{equation}
\|B_{j, m}^{\varphi}(f, g)\textbf{1}_{2^m\gamma'(\epsilon_j)^{-1},
n}\|_{1}\leq C2^{-\beta m/4}\|f\|_{2} \|g\|_{2}, \quad  \textrm{if
$|\gamma'(\epsilon_j)|\leq 2^{-\beta m}$},\label{estimate4-2-1}
\end{equation}
and
\begin{equation}
\|B_{j, m}^{\varphi}(f, g)\textbf{1}_{2^m\gamma'(\epsilon_j)^{-1},
n}\|_{1}\leq C2^{m/4}|\gamma'(\epsilon_j)|^{1/2}\|f\|_{2} \|g\|_{2},
\quad \textrm{if $|\gamma'(\epsilon_j)|\geq 2^{-\beta
m}$}.\label{estimate4-2-3}
\end{equation}

\subsection{Part 3: Conclusion}

If $|\gamma'(\epsilon_j)|\geq 2^{-\beta m}$, balancing the
\eqref{4-1-estimate2} with \eqref{estimate4-2-3} yields
\begin{align*}
\|B_{j, m}^{\varphi}(f, g)\textbf{1}_{2^m\gamma'(\epsilon_j)^{-1},
n}\|_{1}&\leq C\min\{(2^m|\gamma'(\epsilon_j)|)^{-1/6},
2^{m/4}|\gamma'(\epsilon_j)|^{1/2}\}\|f\|_{2} \|g\|_{2}\\
&\leq C 2^{-m/16}\|f\|_{2} \|g\|_{2}.
\end{align*}
If $|\gamma'(\epsilon_j)|\leq 2^{-\beta m}$, the
\eqref{estimate4-2-1} is already good enough. This finishes the
proof of Proposition \ref{prop4-1}.

%
%
%
%
%
%
%
%
%
%

\section{Estimate of $\|B_{j, m}^{\Phi}(f, g)\|_{1}$}
\label{sec5}

Let $\widehat{\Phi}\in C_{0}^{\infty}(\mathbb{R})$ be supported in
$\{\xi\in\mathbb{R} : 1/2\leq |\xi|\leq 2\}$ and $B_{j, m}^{\Phi}(f,
g)$ be as defined at the beginning of Section \ref{main-estimates}
(with $\varphi$ there replaced by $\Phi$).

\begin{proposition}\label{prop4-2}
Assume that $\Gamma(t)=(t, \gamma(t))\in \textbf{F}(-1, 1)$
(\footnote{We do not need the condition \eqref{mt} for this
proposition.}). For  any $\beta<1$, there exist an $L\in\mathbb{N}$
and a constant $C'_{\beta}$  such that whenever $j\geq L$,
$m\in\mathbb{N}$, and $f, g\in L^2(\mathbb{R})$ we have
\begin{equation}
\|B_{j, m}^{\Phi}(f, g)\|_{1}\leq C'_{\beta}A_{j, m}^{\beta}
\|f\|_{2} \|g\|_{2},          \label{lemmaA-2}
\end{equation}
where $A_{j, m}$ equals $C_{j, m}$ if $|\gamma'(\epsilon_j)|\leq
K_1|\epsilon_j|^{c_1}$ and $C_{j, m}'$ if $|\gamma'(\epsilon_j)|\geq
K_2|\epsilon_j|^{-c_1}$ (with $C_{j, m}$ and $C_{j, m}'$ defined as
in Proposition \ref{prop4-1}).
\end{proposition}

\begin{remark}
This proposition is a consequence of Proposition \ref{prop4-1}. It
is essentially the Lemma 5.1 contained in the arXiv preprint
(arXiv:0805.0107) (which was later published as Li~\cite{Li2013}).
\end{remark}

\begin{proof}[Proof of Proposition \ref{prop4-2}]
We only prove the case when $|\gamma'(\epsilon_j)|\leq
K_1|\epsilon_j|^{c_1}$ while the other case can be handled
similarly. Let $\phi$ be a Schwartz function on $\mathbb{R}$ such
that $\int \!\phi =1$ and $\supp\widehat{\phi}\subset [-1/100,
1/100]$. Denote $\phi_K(x)=K^{-1}\phi(K^{-1}x)$. We have
\begin{align*}
&B_{j, m}^{\Phi}(f, g)(x)=|\gamma'(\epsilon_j)|^{1/2} \sum_{n\in
\mathbb{Z}}\sum_{k_1, k_2\in \mathbb{Z}} \int_{0}^{\infty} \!
(\textbf{1}_{2^m, n+k_1}*\phi_{2^m} \cdot f*\Phi)(\gamma'(\epsilon_j)x-2^m t)\\
&\quad \cdot (\textbf{1}_{2^m\gamma'(\epsilon_j)^{-1},
n+k_2}*\phi_{2^m\gamma'(\epsilon_j)^{-1}} \cdot g*\Phi)(x-2^m
Q_{\epsilon_j}(t))\rho(t) \,\mathrm{d}t \cdot
\textbf{1}_{2^m\gamma'(\epsilon_j)^{-1}, n}.
\end{align*}

We then make the decomposition $B_{j, m}^{\Phi}(f,
g)(x):=\textrm{I}+\textrm{II}$ by splitting the inner summation for
$k_1, k_2$ into two parts such that the first part, denoted by I,
sums over $\{k_1, k_2\in \mathbb{Z} : \max\{|k_1|, |k_2|\}\geq A\}$
and the second one, denoted by II, over $\{k_1, k_2\in \mathbb{Z} :
\max\{|k_1|, |k_2|\}< A\}$ with $A=C_{j, m}^{-(1-\beta)/2}>1$.

Using the fast decay of $\textbf{1}_{2^m, n+k_1}*\phi_{2^m}$ and
$\textbf{1}_{2^m\gamma'(\epsilon_j)^{-1},
n+k_2}*\phi_{2^m\gamma'(\epsilon_j)^{-1}}$ yields that
\begin{align*}
|\textrm{I}|&\leq C|\gamma'(\epsilon_j)|^{1/2}\sum_{\substack{k_1,
k_2\in \mathbb{Z}\\ \max\{|k_1|, |k_2|\}\geq A}}\\
&\quad \int_{0}^{\infty} \! \frac{|f*\Phi(\gamma'(\epsilon_j)x-2^m
t) g*\Phi(x-2^m
Q_{\epsilon_j}(t))\rho(t)|}{(1+|t+k_1|)^{N_1}(1+|\gamma'(\epsilon_j)Q_{\epsilon_j}(t)+k_2|)^{N_2}}
\,\mathrm{d}t\\
&\leq C(A^{1-N_1}+A^{1-N_2})|\gamma'(\epsilon_j)|^{1/2}\cdot \\
&\quad \int_{0}^{\infty} \! |f*\Phi(\gamma'(\epsilon_j)x-2^m t)
g*\Phi(x-2^m Q_{\epsilon_j}(t))\rho(t)| \,\mathrm{d}t
\end{align*}
for any $N_1, N_2\in \mathbb{N}$.

By H\"older's and Young's inequalities, we get
\begin{equation}
\|\textrm{I}\|_1\leq C(A^{1-N_1}+A^{1-N_2})\|f\|_2\|g\|_2\leq
C''_{\beta}C_{j, m}^{\beta}\|f\|_2\|g\|_2,\label{boundI}
\end{equation}
where the second inequality holds whenever $N_1, N_2\geq
(1+\beta)/(1-\beta)$.

On the other hand, since
\begin{equation*}
\supp (\mathcal{F}[\textbf{1}_{2^m, n+k_1}*\phi_{2^m} \cdot
f*\Phi])\subset \{\xi\in\mathbb{R} : 3/8\leq |\xi|\leq 17/8\}
\end{equation*}
and
\begin{equation*}
\supp (\mathcal{F}[\textbf{1}_{2^m\gamma'(\epsilon_j)^{-1},
n+k_2}*\phi_{2^m\gamma'(\epsilon_j)^{-1}} \cdot g*\Phi])\subset
\{\xi\in\mathbb{R} : 3/8\leq |\xi|\leq 17/8\},
\end{equation*}
we then have
\begin{align*}
\textrm{II}&=\sum_{n\in \mathbb{Z}}\sum_{\substack{k_1, k_2\in
\mathbb{Z}\\ \max\{|k_1|, |k_2|\}< A}}B_{j, m}^{\varphi}(\textbf{1}_{2^m, n+k_1}*\phi_{2^m} \cdot f*\Phi,\\
&\qquad \textbf{1}_{2^m\gamma'(\epsilon_j)^{-1},
n+k_2}*\phi_{2^m\gamma'(\epsilon_j)^{-1}} \cdot
g*\Phi)(x)\textbf{1}_{2^m\gamma'(\epsilon_j)^{-1}, n}(x).
\end{align*}

Using Proposition \ref{prop4-1} and the Cauchy-Schwarz inequality,
we have
\begin{equation}
\|\textrm{II}\|_1\leq C_{\beta}C_{j,
m}A^2\|f\|_2\|g\|_2=C_{\beta}C_{j,
m}^{\beta}\|f\|_2\|g\|_2.\label{boundII}
\end{equation}
The desired inequality \eqref{lemmaA-2} follows from \eqref{boundI}
and \eqref{boundII}.
\end{proof}

%
%
%
%
%
%
%
%
%
%

\section{Proof of Theorem \ref{main-theorem}}\label{near-origin}

We prove Theorem \ref{main-theorem} in this section. Let $\rho\in
C_{0}^{\infty}(\mathbb{R})$ be an odd function supported in $\{t\in
\mathbb{R} : 1/2\leq |t| \leq 2 \}$ and $\rho_j(t) =2^j\rho(2^j t)$
such that
\begin{equation*}
1/t = \sum_{j\in \mathbb{Z}}\rho_j(t), \quad \textrm{if $t\neq 0$}.
\end{equation*}
Then
\begin{equation*}
H_{\Gamma}(f, g)(x)=\sum_{j\geq 0} \int_{-1}^{1} \!
f(x-t)g(x-\gamma(t))\rho_j(t) \,\mathrm{d}t.
\end{equation*}
Let $L\in \mathbb{N}$. If $0\leq j\leq L$, we can trivially estimate
the $L^1$-norm of each summand above by H\"older's inequality and
get a bound in the form of $C\|f\|_2\|g\|_2$. Hence we may assume
$j> L$ below. By the Fourier inversion formula we need to estimate
\begin{equation*}
\widetilde{H}_{\Gamma}(f, g)(x)= \sum_{j> L} \iint \!
\widehat{f}(\xi) \widehat{g}(\eta) \mathfrak{m}_j(\xi, \eta) e^{2\pi
i (\xi+\eta)x} \,\mathrm{d}\xi \,\mathrm{d}\eta,
\end{equation*}
where
\begin{equation}\label{mj}
\mathfrak{m}_j(\xi, \eta)=\int_{\mathbb{R}} \! \rho(t)e^{-2\pi
i(2^{-j}\xi t+\eta \gamma(2^{-j}t))}\,\mathrm{d}t.
\end{equation}

Let $\widehat{\Phi}\in C_{0}^{\infty}(\mathbb{R})$ be an even
nonnegative function supported in $\{\xi\in\mathbb{R} : 1/2\leq
|\xi|\leq 2\}$ such that
\begin{equation*}
\sum_{m\in \mathbb{Z}}\widehat{\Phi}(\frac{\xi}{2^m})=1, \quad
\textrm{if $\xi\neq 0$}.
\end{equation*}
Let $m, m'\in \mathbb{Z}$. Set
\begin{align*}
\mathfrak m_{j,m,m'}(\xi,\eta)= \widehat{\Phi}(\frac{\xi}{2^{j+m}})
\widehat{\Phi}(\frac{\eta}{2^{m'}\Delta_j})\mathfrak m_j(\xi,\eta),
\end{align*}
where $\Delta_j$ is defined in Section \ref{statement-theorem}. Then
 $\mathfrak m_j(\xi,\eta)$  can be decomposed as the sum of
\begin{align*}
\mathfrak m_{j,+,+}(\xi,\eta)  &= \sum_{\substack{m,m' \geq 0\\
|m'-m|< C}}  \mathfrak m_{j,m,m'}(\xi,\eta),
\\
\mathfrak m_{j,-,-}(\xi,\eta)  &= \sum_{m< 0}\quad \sum_{m'< 0}
\,\quad\mathfrak m_{j,m,m'}(\xi,\eta),
\\
\mathfrak m_{j,+,-}(\xi,\eta)  &= \sum_{m \geq 0}\sum_{m'\leq m-C}
\,\,\,\,\!\mathfrak m_{j,m,m'}(\xi,\eta),
\end{align*}
and
\begin{equation*}
\mathfrak m_{j,-,+}(\xi,\eta)  = \sum_{m' \geq 0}\sum_{m\leq
m'-C}\,\, \mathfrak m_{j,m,m'}(\xi,\eta),
\end{equation*}
where $C$ is a large constant (to be determined later; see
\eqref{lg80} below). Then
\begin{align*}
\widetilde{H}_{\Gamma}(f, g)(x)&= \sum_{(*,**)\in \mathcal A}
\left(\sum_{j> L}\iint \! \widehat{f}(\xi)\widehat{g}(\eta)
\mathfrak m_{j,*,**}(\xi,\eta) e^{2\pi i(\xi+\eta)x}
\,\mathrm{d}\xi\,\mathrm{d}\eta\right)
\\
&=:\sum_{(*,**)\in\mathcal A}  \widetilde{H}_{(*,**)}(f, g)(x),
\end{align*}
where the index set $\mathcal A$ is given by
\begin{align}\label{pm6}
\mathcal A = \{(+,+), (-,-),(+,-), (-,+)\}.
\end{align}
We split $\widetilde{H}_{\Gamma}(f, g)(x)$ into two parts:

\textbf{Major part}: $(*,**)= (+,+)$;

\textbf{Minor part}: $(*,**) =(-,-)$, $(+,-)$, and $(-,+)$.

The essential difficulty in the proof of Theorem \ref{main-theorem}
lies in the estimates of the major part. All our preparations in
Section \ref{phase-function}-\ref{sec5} are done for it. The minor
part can be reduced to classical paraproducts by using the Taylor
and Fourier series expansions, and then handled by Theorem
\ref{para}.

The following proposition completes the proof of Theorem
\ref{main-theorem}.
\begin{proposition}\label{mm009}
Using the notations above, we have
\begin{enumerate}[(i)]
\item {Major part}: if $(*,**)=(+,+)$, then
\begin{align*}
\| \widetilde{H}_{(*,**)}(f, g)\|_1 \leq C\|f\|_2\|g\|_2.
\end{align*}
\item Minor part: if $(*,**)\neq  (+,+)$, then
\begin{align*}
\| \widetilde{H}_{(*,**)}(f, g)\|_{r} \leq C\|f\|_{p_1}\|g\|_{p_2},
\end{align*}
for all $p_1>1$ and $p_2>1$ such that $1/r= 1/p_1+1/p_2$.
\end{enumerate}
\end{proposition}
The rest of this section is devoted to the proof of this
proposition.
%
%
%
%
%
%
%
%
%
%
%
%
%
%
%
%
%
%
%
%
%
%
%
%
%
%
%
%
%

\subsection{Estimates of the major part}\label{mp08}

We consider the case $m, m'\in \mathbb{N}$ with $|m-m'|\leq C$ in
this subsection. Actually it suffices to prove the special case
$m'=m\in \mathbb{N}$ to which we can easily reduce the case $m'=m+b$
(for each $1\leq |b|< C$) simply by replacing $\gamma$ by a constant
multiple of $\gamma$. We also notice that there is a symmetry
between the two cases: $t\geq 0$ and $t\leq 0$ and they can be
handled similarly.

With these simplifications it suffices to prove
\begin{equation}
\bigg\|\sum_{j> L} \sum_{m\in \mathbb{N}} T_{j, m}(f,
g)\bigg\|_1\leq C\|f\|_2\|g\|_2,\label{6-1}
\end{equation}
where
\begin{equation*}
T_{j, m}(f, g)(x)=\iint \! \widehat{f}(\xi) \widehat{g}(\eta)
\widehat{\Phi}(\frac{2^{-j}\xi}{2^m})
\widehat{\Phi}(\frac{2^{-j}\gamma'(2^{-j})\eta}{2^{m}})
\mathfrak{m}^{+}_j(\xi, \eta)e^{2\pi i (\xi+\eta)x} \,\mathrm{d}\xi
\,\mathrm{d}\eta
\end{equation*}
with
\begin{equation*}
\mathfrak{m}^{+}_j(\xi, \eta)=\int_{0}^{\infty} \! \rho(t)e^{-2\pi
i(2^{-j}\xi t+\eta \gamma(2^{-j}t))}\,\mathrm{d}t.
\end{equation*}

To prove \eqref{6-1} we first apply a change of variables and get
\begin{align*}
&T_{j, m}(f, g)(x)=2^{2(j+m)}|\gamma'(2^{-j})|^{-1} \cdot  \\
&\quad \iint \! \widehat{f_{j, m}}(\xi)\widehat{\Phi}(\xi)
\widehat{g_{j, m}}(\eta)\widehat{\Phi}(\eta)  e^{2\pi i
(2^{j+m}/\gamma'(2^{-j}))(\gamma'(2^{-j})\xi+\eta)x}K_{j, m}(\xi,
\eta) \,\mathrm{d}\xi \,\mathrm{d}\eta,
\end{align*}
where
\begin{equation*}
f_{j, m}(x)=2^{-j-m}f(2^{-j-m}x),
\end{equation*}
\begin{equation*}
g_{j, m}(x)=2^{-j-m}\gamma'(2^{-j})g(2^{-j-m}\gamma'(2^{-j})x),
\end{equation*}
and
\begin{equation*}
K_{j, m}(\xi, \eta)=\mathfrak{m}^{+}_j(2^{j+m}\xi,
2^{j+m}\eta/\gamma'(2^{-j}))=\int_{0}^{\infty} \! \rho(t)e^{-2\pi i
2^m(\xi t+\eta Q_{2^{-j}}(t))}\,\mathrm{d}t
\end{equation*}
with
\begin{equation*}
Q_{2^{-j}}(t)=\gamma(2^{-j}t)/(2^{-j}\gamma'(2^{-j})).
\end{equation*}

Let $\textbf{1}_0$ be the indicator function of $\{x\in\mathbb{R} :
1/2\leq |x|\leq 2\}$. Then
\begin{align*}
\|T_{j, m}(f, g)\|_1&=2^{j+m}|\gamma'(2^{-j})|^{-1/2}\|B_{j,
m}^{\Phi}(\mathcal{F}^{-1}[\widehat{f_{j, m}}\textbf{1}_0],
\mathcal{F}^{-1}[\widehat{g_{j, m}}\textbf{1}_0])\|_1,\\
  &\leq C2^{-m/32}2^{j+m}|\gamma'(2^{-j})|^{-1/2}\|\widehat{f_{j, m}}\textbf{1}_0\|_2\|\widehat{g_{j,
  m}}\textbf{1}_0\|_2,\\
  &\leq C2^{-m/32}\|\widehat{f}(\cdot)\textbf{1}_0(2^{-j-m}\cdot)\|_2\|\widehat{g}(\cdot)\textbf{1}_0(2^{-j-m}\gamma'(2^{-j})\cdot)\|_2,
\end{align*}
where we have applied Proposition \ref{prop4-1} and \ref{prop4-2} if
 $L$ is sufficiently large. Thus,
\begin{align*}
\bigg\|\sum_{j> L} &\sum_{m\in \mathbb{N}} T_{j, m}(f, g)\bigg\|_1\\
   &\leq C\sum_{m\in \mathbb{N}}2^{-m/32}\sum_{j>
L}\|\widehat{f}(\cdot)\textbf{1}_0(2^{-j-m}\cdot)\|_2\|\widehat{g}(\cdot)\textbf{1}_0(2^{-j-m}\gamma'(2^{-j})\cdot)\|_2\\
   &\leq C \|f\|_2\|g\|_2.
\end{align*}
In the last inequality, we have used the Cauchy-Schwarz inequality
and the bound
\begin{equation*}
\sum_{j> L}\textbf{1}_0(2^{-j-m}\gamma'(2^{-j})\eta)\leq C,
\end{equation*}
which follows from the condition \eqref{mt} (and also Remark
\ref{remark2-1} (3)). This finishes the estimates of the major part.

%
%

\subsection{Estimates of the minor part}\label{minor09}

For the minor part, we begin with $(*,**) =(-,-)$. Notice
\begin{align*}
\mathfrak m_{j,-,-}(\xi,\eta) &=\sum_{m=-\infty}^{m'}\mathfrak
m_{j,m,m'}(\xi,\eta) + \sum_{m'=-\infty}^{m-1}\mathfrak
m_{j,m,m'}(\xi,\eta)
\\
&= :\mathfrak m_{j,-,m'}(\xi,\eta)+\mathfrak m_{j,m,-}(\xi,\eta)
\end{align*}
The treatments of $\mathfrak m_{j,-,m'}(\xi,\eta)$ and $\mathfrak
m_{j,m,-}(\xi,\eta)$ are similar. We show how to handle $\mathfrak
m_{j,-,m'}(\xi,\eta)$. Set
\begin{equation*}
\widehat\Psi(\xi)=\sum\limits_{m\leq
0}\widehat\Phi(\frac{\xi}{2^m}).
\end{equation*}
By employing the Taylor expansion, $\mathfrak m_{j,-,m'}(\xi,\eta)$
is equal to
\begin{align*}
 & \sum_{p=0}^\infty \sum_{q=0}^{\infty}\frac{1}{p!}\frac{1}{q!}
\widehat\Psi(\frac{\xi}{2^{j+m'}})\widehat\Phi(\frac{\eta}{
2^{m'}\Delta_j})
 \int\! \rho (t)(-2\pi i \xi 2^{-j}t)^{p}(-2\pi i\eta \gamma (2^{-j}t))^q \,\mathrm{d}t
  \\
  =& \sum_{p=0}^\infty \sum_{q=0}^{\infty}c_{j,p,q}\frac{1}{p!}\frac{1}{q!} 2^{m'(p+q)}
 \mathcal N_{j,p,q}(\xi,\eta),
\end{align*}
where
\begin{equation}\label{c95}
c_{j,p,q} =  
\int\! \rho (t)(-2\pi i t)^{p}   \big (-2\pi i
\gamma(2^{-j}t)\Delta_j) \big)^q \,\mathrm{d}t
\end{equation}
and
\begin{equation*}
\mathcal
N_{j,p,q}(\xi,\eta)=\widehat\Psi(\frac{\xi}{2^{j+m'}})(\frac{\xi}{2^{j+m'}})^p\widehat\Phi(\frac{\eta}{
2^{m'}\Delta_j})
 (\frac{\eta}{ 2^{m'}\Delta_j})^q.
\end{equation*}
Since $\rho$ is an odd function, $c_{j,0,0}=0$ for all $j\in\mathbb
Z$ and thus we do not need to consider
$\mathcal N_{j,0,0}(\xi,\eta)$ for all $j>L$.
This yields a decay factor as follows
\begin{equation*}
2^{m'(p+q)} \leq 2^{m'} \quad{\rm if}\quad (p,q) \neq (0,0),
\end{equation*}
which allows us to sum over $m'\leq 0$ later.

The condition \eqref{origin-condition1} gives
\begin{equation*}
|c_{j,p,q}|\leq \|\rho\|_1(4\pi)^p(2\pi C_1)^q
\end{equation*}
and
\begin{equation*}
\sum_{p,q \geq 0 }\frac  1 {p!} \frac 1  {q!} |c_{j,p,q}| < C<\infty
\end{equation*}
for some constant $C$ which is independent of $j$.

Set
\begin{align*}
\mathcal N_{p,q}(\xi,\eta) =\sum_{j >L} \mathcal N_{j,p,q}(\xi,\eta)
=
\sum_{j>L}\widehat\Psi(\frac{\xi}{2^{j+m'}})(\frac{\xi}{2^{j+m'}})^p\widehat\Phi(\frac{\eta}{
2^{m'}\Delta_j })
 (\frac{\eta}{ 2^{m'}\Delta_j })^q.
\end{align*}
It suffices to show that $\mathcal N_{p,q}$, as a bilinear
multiplier, maps $L^{p_1}\times L^{p_2}$ to $L^r$ with a bound
independent of ${m'}$. Indeed, the dependence of ${m'}$ can be
removed easily via the following claim:
\begin{claim}\label{cl07}
Assume $\mathcal M(\xi,\eta)$, as a symbol for a bilinear
multiplier, maps $L^{p_1}\times L^{p_2}$ to $L^r$ with a bound $A$.
Here $p_1>1$, $p_2>1$, and $ 1/p_1+1/p_2=1/r. $
Let $R>0$ be any constant. Then $\mathcal M_R(\xi,\eta)=\mathcal
M(R\xi,R\eta)$
is also a bounded bilinear multiplier which maps $L^{p_1}\times L^{p_2}$ to $L^r$ with the same bound $A$. 
\end{claim}
Claim \ref{cl07} can be proved by a standard rescaling argument and
we omit the details here.

Applying the same arguments to $\mathfrak m_{j,m,-}(\xi,\eta)$, we
obtain the corresponding multiplier
\begin{equation*}
\tilde{\mathcal N}_{p,q}(\xi,\eta) = \sum_{j>L}
\widehat\Phi(\frac{\xi}{2^{j+m-1}})(\frac{\xi}{2^{j+m-1}})^p
\widehat\Psi(\frac \eta {2^{m-1}\Delta_j })(\frac \eta
{2^{m-1}\Delta_j })^q.
\end{equation*}
Again, Claim \ref{cl07} allows us to dispose the factor $2^{m-1}$ on
the right hand side.

To sum up, the case $(*,**)=(-,-)$ is reduced to the problem
concerning the boundedness of the bilinear multipliers whose symbols
are given by
\begin{equation*}
\sum_{j>L}\widehat\Psi(\frac{\xi}{2^{j}})(\frac{\xi}{2^{j}})^p\widehat\Phi(\frac{\eta}{
\Delta_j })(\frac{\eta}{ \Delta_j })^q
\end{equation*}
and
\begin{equation*}
\sum_{j>L
}\widehat\Phi(\frac{\xi}{2^{j}})(\frac{\xi}{2^{j}})^p\widehat\Psi(\frac{\eta}{
\Delta_j })(\frac{\eta}{ \Delta_j })^q,
\end{equation*}
which is covered by Theorem \ref{para} as a special case when
$(n_1,n_2)=(0,0)$.


Now we turn to $(*,**)=(-,+)$. Notice
\begin{align*}
\mathfrak m_{j,-,+} (\xi,\eta)=\sum_{m'\geq 0}  \mathfrak
m_{j,-,m'}(\xi,\eta)
\end{align*}
where
\begin{align*}
\mathfrak m_{j,-,m'}(\xi,\eta) = \sum_{m=-\infty}^{m'-C}\mathfrak
m_{j,m,m'}(\xi,\eta).
\end{align*}
Applying the Fourier series of $\mathfrak m_{j}(\xi,\eta)$ yields
\begin{align*}         
&\mathfrak m_{j,-,m'}(\xi,\eta)
\\
= &
\widehat\Psi(\frac{\xi}{2^{j+{m'}-C}})\widehat\Phi(\frac{\eta}{2^{m'}
\Delta_j}) \mathfrak{m}_j(\xi, \eta)
\\
=&
\widehat\Psi(\frac{\xi}{2^{j+{m'}-C}})\widehat\Phi(\frac{\eta}{2^{m'}
\Delta_j}) \sum_{n_1,n_2\in \mathbb Z} a_{n_1,n_2} e^{2\pi i (n_1
\frac \xi {2^{j+{m'}-C}}+n_2 \frac \eta {2^{m'}\Delta_j})}.
\end{align*}
Since
\begin{align}\label{lg80}
\left|D_t\left(\eta \gamma (2^{-j}t)\right)\right|  = \left |\frac
\eta {\Delta_j}\right | >2^9 |D_t(\xi 2^{-j}t)|
\end{align}
given $C$ sufficiently large, integration by parts gives the
following fast decay
\begin{equation*}
|a_{n_1,n_2}| \leq C_N (1+n_1^2+n_2^2)^{-N} 2^{-{m'}N}\quad
\textrm{for any}\quad N\in\mathbb N.
\end{equation*}
Consequently, we only need to deal with a fixed ${m'}$ and a fixed pair $(n_1,n_2)$. 
Set
\begin{equation}\label{good1}
\mathcal N_{n_1,n_2} (\xi,\eta) = \sum_{j>L}
\widehat\Psi(\frac{\xi}{2^{j+{m'}-C}})e^{2\pi i n_1 \frac \xi
{2^{j+{m'}-C}}} \widehat\Phi(\frac{\eta}{2^{m'} \Delta_j}) e^{2\pi
in_2 \frac \eta {2^{m'}\Delta_j}}.
\end{equation}
Similarly, in the case $(*,**)=(+,-)$ we need to handle the
following multiplier
\begin{equation}\label{good2}
\tilde{\mathcal N}_{n_1,n_2} (\xi,\eta)=
 \sum_{j>L} \widehat\Phi(\frac{\xi}{2^{j+m}})e^{2\pi i n_1 \frac \xi {2^{j+m}}}
\widehat\Psi(\frac{\eta}{2^{m-C} \Delta_j}) e^{2\pi in_2 \frac \eta
{2^{m-C}\Delta_j}}.
\end{equation}
By Claim \ref{cl07}, the factor $2^{m'}$ in \eqref{good1} and the
 factor $2^m$ in \eqref{good2} are disposable. Thus we reduce our
problem to establishing the boundedness of the paraproducts whose
symbols are given by
\begin{equation}\label{cc1}
 \sum_{j>L} \widehat\Psi(\frac{\xi}{2^{j-C}})e^{2\pi i n_1 \frac \xi {2^{j-C}}}
\widehat\Phi(\frac{\eta}{ \Delta_j}) e^{2\pi in_2 \frac \eta
{\Delta_j}}
\end{equation}
and
\begin{equation}\label{cc2}
\sum_{j>L} \widehat\Phi(\frac{\xi}{2^{j}})e^{2\pi i n_1 \frac \xi
{2^{j}}} \widehat\Psi(\frac{\eta}{2^{-C} \Delta_j}) e^{2\pi in_2
\frac \eta {2^{-C}\Delta_j}},
\end{equation}
with bounds whose growth rates are at most a polynomial of
$(1+n_1^2+n_2^2)$, which also follows from Theorem \ref{para}.
Indeed, Theorem \ref{para} is applicable to the multiplier
\eqref{cc2}, since the sequence $\{2^{-C}\Delta_j\}_{j>L}$ satisfies
the condition \eqref{dl}. For the multiplier \eqref{cc1}, one can
perform a rescaling argument $(\xi,\eta)\to (2^{-C}\xi, 2^{-C}\eta)$
and notice the sequence $\{2^C\Delta_j\}_{j>L}$ satisfies the
condition \eqref{dl}, which allows us to apply Theorem \ref{para}
again. This finishes the estimates of the minor part.

%
%
%
%
%
%
%
%
%
%
\section{The bilinear maximal functions}\label{max77}

This section is devoted to the proof of Theorem \ref{maxT}.
The arguments we use here are essentially those from \cite[Section
7]{Li-Xiao}. Recall that
\begin{align*}
M_\Gamma(f,g)(x) =\sup_{0<\epsilon <1} \epsilon^{-1} \int_{0
}^\epsilon\! f(x-t)g(x-\gamma(t)) \,\mathrm{d}t ,
\end{align*}
where we have assumed that $f$ and $g$ are both nonnegative. We want
to show
\begin{align}\label{mx80}
\| M_\Gamma(f,g) \|_1\leq C\|f\|_2\|g\|_2.
\end{align}
The proof of (\ref{mx80}) is almost identical to the proof of
Theorem \ref{main-theorem} in Section \ref{near-origin}. One
noticeable difference is that the minor part of the bilinear maximal
function can be controlled pointwisely by the Hardy-Littlewood
maximal function.

Let $\rho \in C^{\infty}([1/4, 1])$ be nonnegative with $\rho(1/2)
=1$. Set $\rho_j(t) = 2^{j}\rho(2^jt)$. It suffices to establish the
boundedness of the following maximal function
\begin{equation*}
M^*(f,g) (x)= \sup_{j >L} \int\! f(x-t)g(x-\gamma(t))
\rho_j(t)\,\mathrm{d}t=:  \sup_{j >L}  M_j(f,g)(x),
\end{equation*}
where $L\in\mathbb N$ is sufficiently large.

Like what we did in Section \ref{near-origin}, we have
\begin{align*}
M_j(f,g)(x) &= \iint\! \widehat f(\xi)\widehat g(\eta) \mathfrak
m_{j}(\xi, \eta) e^{2\pi i
(\xi+\eta)x}\,\mathrm{d}\xi\,\mathrm{d}\eta
\\
&= \sum_{(*,**)\in \mathcal A} \iint\!  \widehat f(\xi)\widehat
g(\eta) \mathfrak m_{j,*,**}(\xi,\eta)  e^{2\pi i
(\xi+\eta)x}\,\mathrm{d}\xi\,\mathrm{d}\eta
\\
&= :\sum_{(*,**)\in \mathcal A} M_{j,*,**}(f,g)(x),
\end{align*}
where $\mathfrak m_j(\xi,\eta)$ and  $\mathcal A$ are as defined in
(\ref{mj}) and (\ref{pm6}) respectively. It suffices to prove
\begin{align}\label{mm90}
\left\|\sup_{j}|M_{j,*,**}(f,g)|\right\|_1 \leq C\|f\|_2\|g\|_2
\end{align}
for each pair $(*,**)\in\mathcal A$.
\begin{lemma}[Minor part]\label{lm9}
Let $M(f)$ denote the Hardy-Littlewood maximal function of $f$. If
$(*,**)\neq (+,+)$, then there is a constant $C>0$ such that
\begin{align*}
\sup_{j>L} |M_{j,*,**}(f,g)(x)| \leq CM(f)(x)M(g)(x).
\end{align*}
\end{lemma}
Assume this lemma for a moment. The Cauchy-Schwarz inequality and
the Hardy-Littlewood maximal function yield
\begin{align*}
\left\|\sup_{j>L} |M_{j,*,**}(f,g)|\right\|_{r} & \leq C\|
M(f)M(g)\|_{r}
\\
&\leq C\|M(f)\|_{p_1}\|M(g)\|_{p_2}
\\&
 \leq C \|f\|_{p_1}\|g\|_{p_2}.
\end{align*}
In particular, if $(*,**)\neq (+,+)$, (\ref{mm90}) follows from the
 inequality above by taking $(r,p_1,p_2) = (1,2,2)$.
\begin{proposition}[Major part]
If $(*,**) = (+,+)$, we have
\begin{align*}
\left\|\sup_{j>L} |M_{j,+,+}(f,g)|\right \|_1 \leq C \|f\|_2\|g\|_2.
\end{align*}
\end{proposition}
This proposition is essentially the result obtained in Subsection
\ref{mp08}. Indeed, we have the following pointwise estimate
\begin{align*}
\sup_{j>L} |M_{j,+,+}(f,g)(x)| \leq  \sum_{j>L} |M_{j,+,+}(f,g)(x)|.
\end{align*}
Then (\ref{6-1}) implies
\begin{align*}
\left\|\sup_{j>L} |M_{j,+,+}(f,g)| \right\|_1\leq \bigg \|
\sum_{j>L} |M_{j,+,+}(f,g)|\bigg\|_1\leq C\| f\|_2\|g\|_2.
\end{align*}

It remains to verify Lemma \ref{lm9}.
We consider the case $(*,**) =(-,-)$ first. Most of the calculation
in Subsection \ref{minor09} remains valid. In particular, we have
\begin{align*}
\mathfrak m_{j,-,-}(\xi,\eta): &= \sum_{p,q\in\mathbb N} c_{j,p,q}
\frac{1}{p!}\frac{1}{q!} \widehat\Psi({\frac{\xi}{2^{j}}} )
(\frac{\xi}{2^j})^p\widehat \Psi(\frac \eta {\Delta_j })(
\frac{\eta}{\Delta_j})^q.
\end{align*}
Notice
\begin{align*}
\sup_{j\in\mathbb N}\left |\int\!\widehat\Psi({\frac{\xi}{2^{j}}} )
(\frac{\xi}{2^j})^p \widehat f(\xi) e^{2\pi i \xi
x}\,\mathrm{d}\xi\right | \leq C_1' M(f)(x)
\end{align*}
and
\begin{align*}
\sup_{j\in\mathbb N}\left|\int\! \widehat \Psi(\frac \eta {\Delta_j
})( \frac{\eta}{\Delta_j})^q \widehat g (\eta)e^{2\pi i \eta
x}\,\mathrm{d}\eta\right | \leq C_2' M(g)(x),
\end{align*}
where $C_1'$ and $C_2'$ depend at most exponentially on $p$ and $q$.
Thus
\begin{align*}
\sup_{j\in\mathbb N} \left | \iint\! \widehat f(\xi) \widehat
g(\eta) \mathfrak m_{j,-,-} (\xi,\eta) e^{2\pi i (\xi+\eta)x}
\,\mathrm{d}\xi\,\mathrm{d}\eta \right| \leq CM(f)(x)M(g)(x),
\end{align*}
which proves Lemma \ref{lm9} when $(*,**)=(-,-)$.

The cases $(-,+)$ and $(+,-)$ are exactly the same. We only show how
to handle the former one. Using the same notations as in Subsection
\ref{minor09}, we have
\begin{align}\label{th00}
\mathfrak m_{j,-,+}(\xi,\eta) =\sum_{m'\geq 0} \mathfrak
m_{j,-,m'}(\xi,\eta)
\end{align}
and
\begin{align}\label{bb_X}
\mathfrak m_{j,-,m'}(\xi,\eta) =
\widehat\Psi(\frac{\xi}{2^{j+{m'}-C}})\widehat\Phi(\frac{\eta}{2^{m'}
\Delta_j}) \sum_{n_1,n_2\in \mathbb Z} a_{n_1,n_2} e^{2\pi i (n_1
\frac \xi {2^{j+{m'}-C}}+n_2 \frac \eta {2^{m'}\Delta_j})},
\end{align}
where
\begin{equation}\label{ee9}
|a_{n_1,n_2}| \leq C_N (1+n_1^2+n_2^2)^{-N} 2^{-{m'}N}\quad
\textrm{for any}\quad N\in\mathbb N.
\end{equation}
Then
\begin{align}\label{to9}
\sup_{j\in \mathbb Z} \left| \int\! \widehat f(\xi)e^{2\pi i \xi x}
\widehat\Psi(\frac{\xi}{2^{j+{m'}-C}})    e^{2\pi i n_1 \frac \xi
{2^{j+{m'}-C}}}           \,\mathrm{d}\xi\right| \leq C(1+n_1^2)
M(f)(x)
\end{align}
and
\begin{align}\label{to8}
\sup_{j\in \mathbb Z} \left| \int\! \widehat g(\eta)e^{2\pi i \eta
x}\widehat\Phi(\frac{\eta}{2^{m'} \Delta_j}) e^{ 2\pi i {n_2 \frac
\eta {2^{m'}\Delta_j}}}
    \,\mathrm{d}\eta\right| \leq C(1+n_2^2) M(g)(x).
\end{align}
In the statements (\ref{to9}) and (\ref{to8}), we have applied the
following fact
\begin{align*}
\sup_{t>0}|f\ast \Omega_t(x-\frac n t)| \leq C_{\Omega}(1+n^2)
M(f)(x),
\end{align*}
where
\begin{align*}
\Omega_t(x) = t\Omega(tx).
\end{align*}
Then (\ref{th00}), (\ref{bb_X}), (\ref{ee9}), (\ref{to9}) and
(\ref{to8}) yield
\begin{align*}
\sup_{j>L} |M_{j,-,+}(f,g)(x)| \leq   C M(f)(x)M(g)(x),
\end{align*}
as desired.

%
%
%
%
%
%
%
%
%
%
%
%
%
%


\subsection*{Acknowledgments}
We would like to express our sincere gratitude to Xiaochun Li for
valuable advice and helpful discussion.



\begin{thebibliography}{99}


\bibitem{CA66}
Carleson, L., \emph{On convergence and growth of partial sums of
Fourier series}, Acta Math. \textbf{116} (1966), 125--157.

\bibitem{CNSW1999}
Christ, M., Nagel, A., Stein, E. M., and Wainger, S., \emph{Singular
and maximal Radon transforms: analysis and geometry}, Ann. of Math.
(2) \textbf{150} (1999), no. 2, 489--577.




\bibitem{Fe73}
Fefferman, C., \emph{Pointwise convergence of Fourier series},
Ann. of Math. (2) \textbf{98} (1973), 551--571

\bibitem{guo}
Guo, J.W., \emph{On lattice points in large convex bodies}, Acta
Arith. \textbf{151} (2012),  no. 1, 83--108.




\bibitem{hormander1973}
H{\"o}rmander, L., \emph{Oscillatory integrals and multipliers on
$FL^p$}, Ark. Mat. \textbf{11} (1973), 1--11.


\bibitem{hormander}
\bysame, \emph{The Analysis of Linear Partial Differential Operators
I}, Springer-Verlag, Berlin, 1983.


\bibitem{LA00}
Lacey, M., \emph{The bilinear maximal functions map into $L^p$ for
$2/3<p\leq 1$}, Ann. of Math. (2) \textbf{151} (2000), no. 1,
35--57.

\bibitem{ML97}
Lacey, M. and Thiele, C.,   \emph{$L^p$ estimates on the bilinear
Hilbert transform for $2<p<\infty$}, Ann. of Math. (2) \textbf{146}
(1997), 693--724.



\bibitem{ML99}
\bysame, \emph{On Calder\'on's conjecture}, Ann. of Math. (2)
\textbf{149} (1999), no. 2, 475--496.


\bibitem{li2008U}
Li, X.,  \emph{Uniform estimates for some paraproducts}, New York J.
Math. \textbf{14} (2008), 145--192.



\bibitem{Li2013}
\bysame, \emph{Bilinear Hilbert transforms along curves, I: the
monomial case}, Anal. PDE \textbf{6} (2013),  no. 1, 197--220.

\bibitem{Li-Xiao}
Li, X. and Xiao, L., \emph{Uniform estimates for bilinear Hilbert
transform and bilinear maximal functions associated to polynomials},
arXiv preprint arXiv:1308.3518 (2013).

\bibitem{Lie2011}
Lie, V., \emph{On the boundedness of the bilinear Hilbert transform
along ``non-flat'' smooth curves}, arXiv preprint arXiv:1110.3517
(2011).

\bibitem{NRW1974-1976}
Nagel, A., Rivi\`{e}re, N., and Wainger, S., \emph{On Hilbert
transforms along curves}, Bull. Amer. Math. Soc. \textbf{80} (1974),
106--108; \emph{On Hilbert transforms along curves. II}, Amer. J.
Math. \textbf{98} (1976), no. 2, 395--403.


\bibitem{NVWW1983}
Nagel, A., Vance, J., Wainger, S., and Weinberg, D., \emph{Hilbert
transforms for convex curves}, Duke Math. J. \textbf{50} (1983), no.
3, 735--744.




\bibitem{SW1978}
Stein, E. M. and Wainger, S., \emph{Problems in harmonic analysis
related to curvature}, Bull. Amer. Math. Soc. \textbf{84} (1978),
no. 6, 1239--1295.


\end{thebibliography}
\end{document}